\documentclass[final,leqno,onefignum,onetabnum]{siamltex1213}

\usepackage{amsmath,amssymb}
\usepackage[caption=false,lofdepth,lotdepth]{subfig}
\usepackage{bbm}

\makeatletter
\renewcommand*\env@matrix[1][*\c@MaxMatrixCols c]{%
  \hskip -\arraycolsep
  \let\@ifnextchar\new@ifnextchar
  \array{#1}}
\makeatother

\title{Low-Rank Approximation and Completion of Positive Tensors\thanks{This work was supported in part by NSF grant CMMI-1450963.}} 

\author{Anil Aswani\thanks{Industrial Engineering and Operations Research, University of California, Berkeley, CA 94720 (aaswani@berkeley.edu).}}

\begin{document}
\maketitle
\slugger{siopt}{xxxx}{xx}{x}{x--x}

\begin{abstract}
Unlike the matrix case, computing low-rank approximations of tensors is NP-hard and numerically ill-posed in general. Even the best rank-1 approximation of a tensor is NP-hard.  In this paper, we use convex optimization to develop polynomial-time algorithms for low-rank approximation and completion of positive tensors.  Our approach is to use algebraic topology to define a new (numerically well-posed) decomposition for positive tensors, which we show is equivalent to the standard tensor decomposition in important cases.  Though computing this decomposition is a nonconvex optimization problem, we prove it can be exactly reformulated as a convex optimization problem.  This allows us to construct polynomial-time randomized algorithms for computing this decomposition and for solving low-rank tensor approximation problems.  Among the consequences is that best rank-1 approximations of positive tensors can be computed in polynomial time.  Our framework is next extended to the tensor completion problem, where noisy entries of a tensor are observed and then used to estimate missing entries.  We provide a polynomial-time algorithm that for specific cases requires a polynomial (in tensor order) number of measurements, in contrast to existing approaches that require an exponential number of measurements.  These algorithms are extended to exploit sparsity in the tensor to reduce the number of measurements needed.  We conclude by providing a novel interpretation of statistical regression problems with categorical variables as tensor completion problems, and numerical examples with synthetic data and data from a bioengineered metabolic network show the improved performance of our approach on this problem.
\end{abstract}



\pagestyle{myheadings}
\thispagestyle{plain}
\markboth{ANIL ASWANI}{APPROXIMATION AND COMPLETION OF POSITIVE TENSORS}

\section{Introduction}

Tensors generalize matrices by describing a multidimensional array of numbers.  More formally, a tensor $\psi$ of \emph{order} $p$ is given by $\psi \in \mathbb{R}^{r_1 \times \cdots \times r_p}$, where $r_i$ is the \emph{dimension} of the tensor in the $i$-th index, for $i = 1,\ldots,p$.  When we would like to refer to a specific entry in the tensor, we use the notation $\psi_x := \psi_{x_1,\ldots,x_p}$, where $x = (x_1,\ldots,x_p)$, $x_i \in [r_i]$ denotes the value of the $i$-th index, and $[s] := \{1,\ldots,s\}$.  Also let $r = \max_i r_i$.  The reasons for choosing this notation will become more clear when discussing our novel interpretation of statistical regression with categorical variables as tensor completion.

The similarity between tensors and matrices is misleading because many problems that are routine and polynomial-time computable for matrices are NP-hard for tensors.  For instance, it is NP-hard to compute the rank of a tensor \cite{hillar2013}, which is defined as the minimal number of rank-1 components needed to represent the tensor: $\text{rank}_\otimes(\psi) = \min \{q\ |\ \psi = \textstyle\sum_{j=1}^q v_1^j \otimes\cdots\otimes v_p^j, \text{ where } v_i^j \in \mathbb{R}^{r_i}\}$, where $\otimes$ is the \emph{tensor product} \cite{landsberg2012,hillar2013}.  Tensor analogs of the matrix singular value decomposition (e.g., \textsc{candecomp}/\textsc{parafac} or \textsc{cp}) are also NP-hard to compute \cite{hillar2013}.  Furthermore, determining the best low-rank approximations for tensors is an ill-posed problem in general \cite{desilva2008}, and computing the best rank-1 approximation is NP-hard in general \cite{hillar2013}.

In this paper, we attack the computational challenges posed by tensor problems by showing that positive tensors are amenable to polynomial-time algorithms with strong guarantees. A new tensor decomposition called a hierarchical decomposition is defined in \S \ref{section:hdpt} using a structure from algebraic topology.  This decomposition is shown to exist, be numerically well-posed, and coincide with the usual tensor \textsc{cp} decomposition \cite{kolda2009,hillar2013} in specific cases.  Section \ref{section:rad} develops a randomized algorithm to compute the hierarchical decomposition in polynomial time that depends on the degrees-of-freedom of the tensor rather than on the total number of tensor entries, which can be exponentially larger than the degrees-of-freedom.  This algorithm can compute a best rank-1 approximation of positive tensors in polynomial time.


Our approach differs from existing methods in a number of ways.  An iterative descent algorithm was proposed for decomposition of positive tensors in \cite{welling2001}, but no convergence guarantees were provided.  Matrix nuclear norm approaches are popular for tensor completion \cite{tomioka2010,signoretto2010,gandy2011,liu2013,mu2014,zhang2014}, though these approaches have exponentially slow statistical convergence.  Using the tensor nuclear norm  (which is NP-hard to compute \cite{friedland2014}) also gives exponentially slow statistical convergence \cite{yuan2015,yuan2016}.  For orthogonal third-order tensors (i.e., $p = 3$), alternating minimization provides polynomial statistical convergence \cite{jain2014}; unfortunately, this guarantee requires using a large number of randomized initializations, and so the results cannot be naturally generalized to higher order tensors without losing polynomial-time computability.  Recently, Lassere hierarchy approaches have been proposed for best rank-1 approximations \cite{nie2014} and tensor completion \cite{rauhut2015}.  These have found success on specific numerical examples, but conditions to guarantee global optimality are currently unavailable.

After presenting our algorithm for computing a hierarchical decomposition, \S \ref{section:tc} extends this framework to the problem of tensor completion \cite{tomioka2010,signoretto2010,gandy2011,liu2013,mu2014,zhang2014,montanari2014}, in which a subset of tensors entries are observed and then used to estimate missing entries.  We provide an algorithm that for specific cases requires a polynomial (in tensor order) number of measurements, which is much lower than the exponential number of measurements required by tensor completion methods using the matrix nuclear norm \cite{tomioka2010,signoretto2010,gandy2011,liu2013,mu2014,zhang2014}.  In the case of a rank-1 tensor, the number of needed measurements of our approach $O((rp^2)^{1+\zeta})$, for any $\zeta > 0$, is essentially a quadratic factor away from the information-theoretic lower bound of $O(rp)$.  Section \ref{section:shd} shows how the algorithms can be improved to exploit sparsity in the tensor.  Numerical examples with synthetic data in \S \ref {section:ne} show our approach outperforms other tensor completion algorithms.  We conclude by providing in \S \ref{section:rcv} a novel interpretation of statistical regression problems with categorical variables as tensor completion problems.  Data from a bioengineered metabolic network is used to show the improved performance of our approach for categorical regression.

\section{Hierarchical Decomposition of Positive Tensors}

\label{section:hdpt}

A structure from algebraic topology \cite{diaconis1998,drton2009} is used to parametrize our new decomposition.  Following the definition of \cite{drton2009}: A \emph{simplicial complex} is a set $\Gamma \subseteq 2^{[p]}$ such that $F \in \Gamma$ and $S \subset F$ implies that $S \in \Gamma$.  The elements of $\Gamma$ are called \emph{faces} of $\Gamma$ and the inclusion-maximal faces are the \emph{facets} of $\Gamma$.  We will assume the facets have been arbitrarily assigned an order, so that we can represent the simplicial complex as $\text{facets}(\Gamma) = \{F_1,F_2,\ldots,F_{m(\Gamma)}\}$, where $m(\Gamma)$ is the number of facets.  We will drop the argument notation $(\Gamma)$ when clear from the context.  Roughly speaking, a simplicial complex is a graph with higher-order connections between vertices.  Whereas edges in a graph can only connect two vertices, facets in a simplicial complex can simultaneously connect an arbitrary number of vertices.  An example of a simplicial complex with zero-, one-, two-, and three-dimensional facets is shown in Figure \ref{fig:sce}.

\begin{figure}
\centering
\includegraphics[trim = 0mm 18.3mm 0mm 17.8mm, clip, scale=0.5]{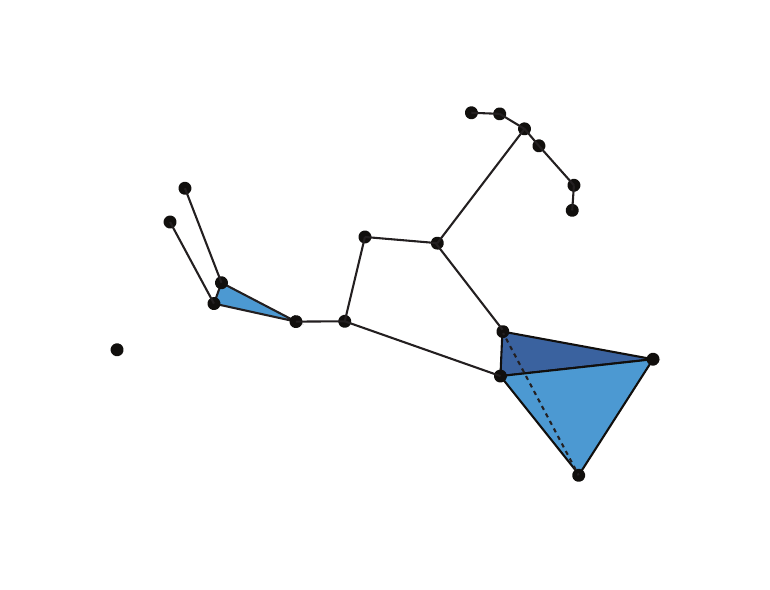}
\caption{\label{fig:sce}Example of a Simplicial Complex}
\end{figure}

This section begins with the definition of a hierarchical decomposition for positive tensors, and this decomposition is parametrized by a simplicial complex.  Hierarchical decompositions are shown to always exist and be numerically well-posed for positive tensors.  Next, we define an important special case of the hierarchical decomposition, which we call a partition decomposition.  This is used to provide instances in which these decompositions exactly coincide with the typical tensor \textsc{cp} decomposition.  The following notation (adapted from \cite{drton2009}) will be needed to write subindices.  Recall the set notation $[s] := \{1,\ldots,s\}$, and define $\mathcal{R} = [r_1]\times\cdots\times[r_p]$.  If $x = (x_1,\ldots,x_p) \in \mathcal{R}$ and $F = \{f_1,f_2,\ldots\} \subseteq [p]$, then $\mathcal{X}_F = (x_{f_1},x_{f_2},\ldots)$, and this vector has the state space $\mathcal{R}_F = [r_{f_1}]\times[r_{f_2}]\times\cdots$.  We use the notation $\mathcal{X}_k = \mathcal{X}_{F_k}$ and $\mathcal{R}_k = \mathcal{R}_{F_k}$ to reduce the number of indices in our equations.


\subsection{Definition}

Motivated by hierarchical log-linear models used in statistics to construct hypothesis tests for contingency tables \cite{drton2009}, we define a \emph{hierarchical decomposition} of a positive tensor to be
\begin{equation}
 \label{eqn:hd}
\psi_x = \prod_{k=1}^m\theta^{(k)}_{\mathcal{X}_k}
\end{equation}
where $\Gamma$ is a simplicial complex with $\text{facets}(\Gamma) = \{F_1,\ldots,F_m\}$, and $\theta^{(k)} \in \mathbb{R}^{r_{f_1}\times r_{f_2}\times\cdots}$ are constants indexed by the different values of $\mathcal{X}_k \in \mathcal{R}_k$.  When $\Gamma$ is such that (\ref{eqn:hd}) is satisfied, we say $\Gamma$ is \emph{correct}; on the other hand, if $\Gamma$ is such that (\ref{eqn:hd}) does not hold, then we say $\Gamma$ is \emph{incorrect}.  To simplify notation, we drop the superscript in $\theta^{(k)}_{\mathcal{X}_k}$ and write this as $\theta_{\mathcal{X}_k}$ when clear from the context.  Also, $\Theta = \{\theta^{(k)} : k = 1,\ldots,m\}$ refers to the set of all parameters.

\subsection{Existence and Representational Complexity}
\label{section:erc}

Existence (and well-posedness) of the hierarchical decomposition of a positive tensor can be shown under a mild boundedness assumption:\\

\noindent \textbf{A1}. The tensor is bounded $M^{-1} \leq \psi_x \leq M$ by some constant $M > 1$.\\

\noindent Our results generalize to the case $M_1 \leq \psi_x \leq M_2$, where $0 < M_1 < M_2$; we keep the above assumption to simplify stating the results.  Relaxing the lower bound to zero is more delicate: In practice, we can choose $M$ sufficiently large such that the lower bound is arbitrarily close to zero.  In theory, relaxing the lower bound to exactly zero requires additional analysis because the loss function we will use, though continuously differentiable, does not have a bounded derivative at zero. 


\begin{proposition}
If $\psi$ satisfies $\mathbf{A1}$, then a hierarchical decomposition of $\psi$ with a correct $\Gamma$ exists.
\end{proposition}

\begin{proof}
The result follows by choosing a simplicial complex: $\text{facets}(\Gamma) = \{F_1\}$, where $F_1 = \{1,\ldots,p\}$, and then setting $\theta_{\mathcal{X}_1} := \psi_x$.$\qquad$
\end{proof}


Note there is a lack of uniqueness of the parametrizing $\Gamma$ because we can always choose a simplicial complex with a single facet, as in the above proof, to specify a valid hierarchical decomposition.  Because of this nonuniqueness, it is useful to define a notation of complexity.  We define the \emph{effective dimension} of a hierarchical decomposition for a specific choice of $\Gamma$ to be $\rho(\Gamma) = \sum_{k=1}^m \prod_{j \in F_k} r_j$.  The effective dimension is the number of coefficients used in the hierarchical decomposition of the tensor.  In many cases, a tensor of low rank can be represented by a hierarchical decomposition with low effective dimension.  Specific examples are given in the next subsection.  Moreover, a counting argument implies that the tensor rank must be upper bounded by the effective dimension: $\text{rank}_{\otimes}(\psi) \leq \rho$.  It is for these reasons we use low effective dimension as a surrogate for low tensor rank when we study the problems of tensor approximation and completion.

\subsection{Numerical Well-Posedness}

Beyond existence, hierarchical decompositions are also well-posed. One of the reasons that computing the best low-rank approximation of a tensor is an ill-posed problem in general \cite{desilva2008} is that though the entries of the tensor might be bounded, the coefficients of the tensor decomposition can be unbounded.  (This can occur because the unbounded nature of the coefficients cancel each other out.)  This leads to unique phenomenon such as having a sequence of tensors of rank two that converge to a tensor of rank three \cite{desilva2008,landsberg2012}.  Fortunately, the situation for nonnegative tensors is better because the approximation problem is well-posed \cite{lim2009,qi2014}.  As we show with the next proposition, the hierarchical decomposition is also well-posed in a particular way that will be important for formulating optimization problems.

\begin{proposition}
\label{proposition:well-posed}
If $\psi$ satisfies $\mathbf{A1}$ and $\Gamma$ is correct, then there exists $\Theta$ such that $M^{-2} \leq \theta_{\mathcal{X}_k} \leq M^{2}$, for all $\mathcal{X}_k \in \mathcal{R}_k$ and $k = 1,\ldots,m$.
\end{proposition}

\begin{proof}
We successively construct a set of parameters and show these satisfy the proposition.  One set of parameters can be defined by performing the following steps:
\begin{enumerate}
\item Set $\mathcal{I} = \mathcal{R}$;
\item For $j = 1,\ldots,m$:
\begin{enumerate}
\item Set $\mathcal{H}_j = \{\mathcal{X}_{j} \in \mathcal{R}_{j} : x \in \mathcal{I}\}$;
\item While $\mathcal{H}_j \neq \emptyset$:
\begin{enumerate}
\item Select an arbitrary element $u \in\mathcal{I}$;
\item Set $\theta_{\mathcal{U}_{k}} = 1$ for all $k = (j+1),\ldots,m$;
\item Set $\theta_{\mathcal{X}_{j}} = \textstyle\psi_x/\prod_{k=1}^{j-1}\theta_{\mathcal{X}_{k}}$, for all $x \in \mathcal{I}$ such that $\mathcal{X}_{k} = \mathcal{U}_{k}, \forall k = (j+1),\ldots,m$;
\item Set $\mathcal{I} = \mathcal{I}\setminus\{x \in \mathcal{I} : \mathcal{X}_{k} = \mathcal{U}_{k}, \forall k = (j+1),\ldots,m\}$;
\item Set $\mathcal{H}_j = \mathcal{H}_j\setminus\{\mathcal{X}_{j} \in \mathcal{H}_j : x \in \mathcal{I} \text{ such that } \mathcal{X}_{k} = \mathcal{U}_{k}, \forall k = (j+1),\ldots,m\}$.
\end{enumerate}
\end{enumerate}
\end{enumerate}
Observe $\mathcal{I}$ lists the subset of indices of $\mathcal{R}$ for which the decomposition is undefined, and $\mathcal{H}_j$ lists the subset of indices of $\mathcal{R}_{j}$ for which the decomposition is undefined.  The intuition behind this algorithm is we successively specify the parameters of the decomposition until there are no indices for which the decomposition is undefined.  The inner loop ensures $\mathcal{H}_j$ becomes empty, and the set $\mathcal{I}$ becomes empty at the end of the algorithm because $\mathcal{I} \equiv \{v \in \mathcal{I} : \mathcal{V}_{k} = \mathcal{U}_{k}, \forall k = (j+1),\ldots,m\}$ when $j = m$.

Next, note that the parameters $\theta_{\mathcal{U}_{k}}$ trivially satisfy $M^{-2} \leq \theta_{\mathcal{U}_{k}} \leq M^{2}$ since $\theta_{\mathcal{U}_{k}} = 1$, and so we only need to show that the remaining parameters satisfy the bounds of the proposition.  For any $j > 1$, suppose that $M^{-1} \leq \prod_{k=1}^{j-1}\theta_{\mathcal{X}_{k}} \leq M$.  If this condition holds, then two consequences follow from step 2.b.iii: (i) $M^{-2} \leq \theta_{\mathcal{X}_{j}} \leq M^2$, and (ii) $M^{-1} \leq \prod_{k=1}^{j}\theta_{\mathcal{X}_{k}} \leq M$.  In fact, for $j=1$ we have that $M^{-1} \leq \theta_{\mathcal{X}_{1}} \leq M$, since $\theta_{\mathcal{X}_{1}} = \psi_x/\prod_{k=1}^{0}\theta_{\mathcal{X}_{k}} = \psi_x$.  This inductively shows that the bounds of the proposition hold for all the remaining parameters.$\qquad$
\end{proof}

This result implies that the parameters $\Theta$ of the decomposition are bounded by an amount that is independent of $p$ and $\Gamma$ as long as the individual entries of the tensor are bounded as in \textbf{A1}.  This will allow us to define constraints in our optimization problems that ensure the numerical scaling of different parameters is controlled.  For numerical reasons, we would like to avoid scalings in which some parameters are very large and other parameters are very small.  This proposition allows us to define constraints that control the scaling.

\subsection{Partition Decomposition}
\label{section:pd}

An important special case of a hierarchical decomposition is when the facets of the simplicial complex $\Gamma$ are a partition of the set $[p]$.  We refer to this instance as a partition decomposition.  The partition decomposition can be written as $\psi_x = \prod_{k = 1}^m\theta_{\mathcal{X}_{k}} = (\theta^{(1)}\otimes\cdots\otimes\theta^{(m)})_{P(x)}$, where the middle equation is the partition decomposition, $\otimes$ is the tensor product, and $P(\cdot)$ is an appropriately-defined permutation of the indices. The partition decomposition is of note because it can be written as the product of tensors with smaller order than $\psi$, and because it exactly coincides in specific cases with a low-rank \textsc{cp} decomposition of tensors.  The \textsc{cp} decomposition is defined as $\psi = \textstyle\sum_{j=1}^q v_1^j \otimes\cdots\otimes v_p^j$, where $v_i^j \in \mathbb{R}^{r_i}$ and $q = \text{rank}_\otimes(\psi)$ is the tensor rank \cite{kolda2009}, and it is a typical tensor decomposition and an analog of the matrix singular value decomposition \cite{kolda2009,hillar2013}.

The simplest case in which the partition decomposition coincides with the \textsc{cp} decomposition is when the partition is given by $\text{facets}(\Gamma) = \{\{1\},\ldots,\{p\}\}$.  In this case, both the partition and \textsc{cp} decompositions represent a rank-1 tensor: $\psi_x = \prod_{k = 1}^p\theta_{x_k} = (\theta^{(1)}\otimes\cdots\otimes\theta^{(p)})_x$, where the $\theta^{(k)} \in \mathbb{R}^{r_k}$ are vectors, the middle equation is the partition decomposition, and the right equation is the \textsc{cp} decomposition.  The decompositions coincide in this case because they are equivalent.

Another instance where the partition and \textsc{cp} decompositions coincide is when the $\theta^{(k)}$ are either vectors or matrices of full rank.  Assume the partitions are arranged so $\theta^{(1)},\ldots,\theta^{(s)}$ are matrices and $\theta^{(s+1)},\ldots,\theta^{(m)}$ are vectors.  Also, let a matrix decomposition of $\theta^{(k)}$ be given by $\theta^{(k)} = \textstyle\sum_{j=1}^{q_k}u_k^{j}\otimes v_k^{j}$, where $q_k$ is the matrix rank of $\theta^{(k)}$, and $u_k^{j},v_k^{j}$ are vectors of appropriate dimensions.  Then we have $\psi_x = \textstyle\prod_{k = 1}^m\theta_{\mathcal{X}_{k}} = \big(\textstyle\sum_{j_1\times \cdots\times j_s\in[q_1]\times\cdots\times[q_s]}u_1^{j_1}\otimes v_1^{j_1}\otimes\cdots \otimes u_s^{j_s}\otimes v_s^{j_s}\otimes\theta^{(s+1)}\otimes\cdots\otimes\theta^{(m)}\big)_{P(x)}$, where the middle and right equations are the partition and \textsc{cp} decompositions, respectively.  The decompositions coincide in this case because the partition decomposition can be used to compute the \textsc{cp} decomposition by computing the matrix singular value decomposition of $\theta^{(k)}$; similarly, the \textsc{cp} decomposition can be used to compute the partition decomposition by computing $\theta^{(k)} = \textstyle\sum_{j=1}^{q_k}u_k^{j}\otimes v_k^{j}$.

\section{Randomized Algorithm for Decompositions and Approximations}

\label{section:rad}

The algorithm in Proposition \ref{proposition:well-posed} implies a hierarchical decomposition can be computed in steps that are polynomial in the number of tensor entries.  However, this computational complexity can be improved with a randomized algorithm that will only need a polynomial in effective dimension $\rho$ number of arithmetic calculations.  This can be a significant improvement because the effective dimension can be much smaller than the number of tensor entries: For instance, a rank-1 tensor has effective dimension $\rho = \sum{r_i}=O(rp)$ while it has $\prod{r_i}=O(r^p)$ entries.


Our approach to developing a randomized algorithm for computing a hierarchical decomposition is to randomly sample entries of the tensor.  With enough samples, the decomposition will have low error with high probability.  In anticipation of generalizing to the tensor completion problem, we allow the sampled entries to be measured with noise.  This noise could be deterministically interpreted as the approximation error of a hierarchical decomposition, meaning the hierarchical decomposition for a specified $\Gamma$ that is closest (as measured by some loss function) to the tensor $\psi$.  As a result, the statistical consequences have deterministic interpretations.  

This section begins by describing the noise and measurement model for sampling entries of the tensor, and then attention turns towards choosing the loss function that will be used to measure the discrepancy between the computed decomposition and the sampled entries.  Specific computational and statistical challenges with choosing the loss function are discussed, and this precludes the use of a squared loss function or of taking the logarithm of the data.  We propose an alternative loss function: This loss has the same minimizer in specific cases as that of the squared loss function, and we show it is majorized and minorized by the squared loss function.  Furthermore, we show this loss function can be minimized in polynomial time by exactly reformulating the optimization problem as a convex program.

Next, we use this reformulation to show an equivalence result between our loss function and the decomposition error as measured by the squared loss.  This equivalence result allows us to study approximation properties using our loss function and then apply the approximation properties to the squared loss.  We use the stochastic processes theory of Rademacher complexity \cite{bartlett2002,kakade2009,lafferty2012,boucheron2013} to bound the approximation error induced by computing a decomposition using a sample of tensor entries (rather than using all the tensor entries).  And the section concludes by presenting a randomized algorithm, which uses the alternative loss function, and proving it has polynomial-time complexity in terms of effective dimension $\rho$.

\subsection{Noise and Measurement Model}


Note we use the indexing notation $\langle i\rangle$ to denote the $i$-th measurement.  For a randomly chosen set of indices $x\langle i\rangle \in \mathcal{R}$, suppose we make a noisy measurement of the corresponding tensor entry $y\langle i\rangle = (1 + z\langle i\rangle)\cdot\psi_{x\langle i\rangle}$, where $z\langle i\rangle$ is noise.  A multiplicative noise model, as opposed to an additive noise model, is used here because this allows us to define a statistical model where measurements $y\langle i\rangle$ are positive-valued while the noise is independent of $x\langle i\rangle$.  However, our results also apply to the case of additive zero-mean noise with the only changes being in the constants of the resulting bounds.  Rather than complicating the presentation, we focus on the multiplicative noise model.  We make the following assumption about the noise:\\

\noindent \textbf{A2}. The noise $z\langle i\rangle$ are iid random variables with a mean of zero $\mathbb{E}(z)=0$, and they are bounded $\mu^{-1} \leq 1+z \leq \mu$ by some constant $\mu > 1$.\\

\noindent The bounds on noise could be relaxed to be unbounded in both directions (i.e., positive and negative).  This is appealing because many interesting noise distributions satisfying the property $\mathbb{E}(z)=0$ are sub-gamma distributions \cite{boucheron2013}.  We do not consider these cases because their consideration does not provide additional theoretical insights; the main difference is slower rates of convergence for heavier-tailed distributions.  And so for simplicity, we assume the above boundedness condition; however, we will use the gamma distribution (which is unbounded) to generate noise for the synthetic data in our numerical examples.

Another note is the reason for choosing a model with $\mathbb{E}(z)=0$ is so $\mathbb{E}[y|x] = \psi_{x}$ holds.  This is a mild assumption because we are interested in computing a decomposition that best approximates $\psi_x$; the $\Theta$ themselves do not have any particular meaning in our decomposition because they are nonunique up to a scaling factor.

We also make an assumption about the measurements available.  For now, we will not impose any conditions on the distribution, except for requiring iid measurements.\\

\noindent \textbf{A3}. The data are iid measurements $(x\langle i\rangle,y\langle i\rangle)$, for $i = 1,\ldots,n$, where $n$ is the number of measurements.

\subsection{Challenges with Choosing Loss Function}

The usual loss function is the squared loss $L(\Theta) = \mathbb{E}((y - \textstyle\prod_{k=1}^m\theta_{\mathcal{X}_k})^2)$,
and when $\Gamma$ is correct the minimizer is given by $\Theta^*$ such that $\psi_x = \prod_{k=1}^m\theta_{\mathcal{X}_k}^*$ \cite{banerjee2005}.  But numerically minimizing this loss is difficult because of nonconvexity of the squared loss in the parameters $\Theta$.  One common approach is to use a heuristic such as alternating least squares (ALS), but this only converges to local optimum \cite{kolda2009}.  

Given the structure of the hierarchical decomposition, it is tempting to compute the decomposition by minimizing $\mathbb{E}((\log y - \textstyle\sum_{k=1}^m\log \theta_{\mathcal{X}_k})^2)$, because this converts the optimization into a linear least squares problem.  However, this a problematic choice because the approach in \cite{banerjee2005} can be used to show that the minimizer of the above loss function is $\mathbb{E}[\log y | x] = \textstyle\sum_{k=1}^m\log \theta_{\mathcal{X}_k} + \mathbb{E}(\log (1+z))$. This is nonideal because the solution will be incorrect by the amount $\mathbb{E}(\log (1+z)) \neq 0$.  Jensen's inequality for concave functions implies $\mathbb{E}(\log (1+z)) \leq \log \mathbb{E}((1+z)) = \log 1 = 0$; so the general case is the nuisance parameter $\mathbb{E}(\log (1+z))$ will be nonpositive.  Taking the exponent $\exp(\mathbb{E}[\log y | x])$ does not resolve the problem because we still have a multiplicative error of $\exp(\mathbb{E}(\log (1+z))) \neq 1$.  

\subsection{Alternative Loss Function}

So if we do not \emph{a priori} know the value of the nuisance parameter $\mathbb{E}(\log (1+z))$, then we could devise a two step procedure that consistently estimates this nuisance parameter and then removes it from the least squares solution, in order to compute a best hierarchical decomposition of $\psi_x$.  We can eliminate the need for considering this nuisance parameter by defining an alternative loss function.  This choice will be subsequently justified by showing that it displays faithful error properties and is amenable to polynomial-time computation.

We use the following loss function
\begin{equation}
\label{eqn:nonconvex_risk}
R(\Theta) = \mathbb{E}\big(-y\cdot\textstyle\sum_{k = 1}^m\log \theta_{\mathcal{X}_k} + \textstyle\prod_{k = 1}^m\theta_{\mathcal{X}_k}\big),
\end{equation}
and the best approximate hierarchical decomposition
\begin{equation}
\label{eqn:nonconvex}
\hat{\Theta} = \arg\min \big\{\hat{R}(\Theta)\ \big|\ \Theta \in \Omega\big\}.
\end{equation}
is defined to be the minimizer of the empirical loss function
\begin{equation}
\hat{R}(\Theta) = \textstyle\frac{1}{n}\sum_{i=1}^n \big(-y\langle i\rangle\cdot\textstyle\sum_{k = 1}^m\log \theta_{\mathcal{X}_k\langle i\rangle} + \textstyle\prod_{k = 1}^m\theta_{\mathcal{X}_k\langle i\rangle}\big)
\end{equation}
subject to the constraint set
\begin{equation}
\label{eqn:nonconvex_constraint}
\Omega = \big\{\Theta : M^{-1} \leq \textstyle\prod_{k = 1}^m \theta_{\mathcal{X}_k} \leq M,\ M^{-2} \leq \theta_{\mathcal{X}_k} \leq M^2,\ \forall x \in \mathcal{R}\big\}.
\end{equation}
We justify this choice by first showing an equivalence to the usual squared loss function.  Our second justification is it enables polynomial-time computation for specific approximation problems for positive tensors that are NP-hard in the case of a general (i.e., not necessarily positive) tensor, and this is shown by rewriting (\ref{eqn:nonconvex}) as a convex optimization problem with a polynomial in $\rho$ and $n$ number of constraints.  

\subsection{Error Properties of Loss Function}

We show the loss function (\ref{eqn:nonconvex_risk}) has favorable error properties.  This loss function resembles the negative log-likelihood for a Poisson distribution: $\frac{1}{n}\sum_{i=1}^n (-y_i\log \mu + \mu)$, where $\mu > 0$ is the rate parameter of the distribution, and this is not surprising because this likelihood can be used to fit hierarchical log-linear models to contingency tables \cite{drton2009}.  Furthermore, maximum likelihood decomposition of nonnegative tensors of count data using the Poisson distribution has been previously considered \cite{chi2012}.  However, this is the wrong interpretation for our case because $\psi_x$ can take continuous (non-integer) values and should not be interpreted as counts in general.

A better interpretation for the loss function (\ref{eqn:nonconvex_risk}) is as a Bregman divergence \cite{banerjee2005}, or more specifically a generalized I-divergence (which is a generalization of the Kullback-Leibler divergence) \cite{banerjee2005,manthey2009}.  This is a more natural interpretation because of the following proposition that shows minimizing either our loss $R(\Theta)$ or the squared loss $L(\Theta)$ recovers the same solution when $\Gamma$ is correct.

\begin{proposition}[Banerjee, et al., 2005 \cite{banerjee2005}]
\label{prop:oracle_equiv}
If $\mathbf{A1}$,$\mathbf{A2}$ hold and $\Gamma$ is correct, then $\Theta^* \in \arg\min \{R(\Theta)\ |\ \Theta \in \Omega \} \Leftrightarrow \Theta^* \in \arg\min \{L(\Theta)\ |\ \Theta \in \Omega \}$.  Moreover, the solution $\Theta^*$ has the property $\psi_x = \prod_{k=1}^m\theta^*_{\mathcal{X}_k}$.
\end{proposition}

A further justification for using the loss (\ref{eqn:nonconvex_risk}) is that it is equivalent to the squared loss in the sense that it both majorizes and minorizes the squared loss.  

\begin{proposition}
\label{prop:equivalence}
Under $\mathbf{A1}$,$\mathbf{A2}$ and for any $\Gamma$, the loss function $R(\Theta)$ majorizes and minorizes the squared loss function $L(\Theta)$, meaning $\textstyle a_l\cdot L(\Theta) + b_l\leq R(\Theta)\leq a_u\cdot L(\Theta) + b_u$, where constants $a_l,a_u > 0$ and $b_l,b_u$ depend on $\mu,M$.
\end{proposition}

\begin{proof}
Define $\mathcal{M} = [(\mu M)^{-1},\mu M]$, and consider the function $f(u) = -y\log(u) + u$ over the domain $u \in \mathcal{M}$.  This function is strongly convex for $u\in\mathcal{M}$, and so we have $-y\log(u) + u \geq -y\log(v) + v + (-y/v + 1)\cdot(u - v) + y\cdot(u-v)/2(\mu M)^2$. Choosing $v = y$ gives $-y\log(u) + u \geq -\mu M\log(\mu M) + (\mu M)^{-1}+(u-y)^2/2(\mu M)^3$.  The lower bound follows by setting $u = \prod_{k=1}^m\theta_{\mathcal{X}_k}$ and taking the expectation of both sides.

The upper bound is shown using the mean-value form of Taylor's theorem, which states that for any $u,v\in \mathcal{M}$: $-y\log(u) + u = -y\log(v)+v + (-y/v + 1)\cdot(u-v) + y\cdot(u-v)^2/2z^2$, for some $z \in \mathcal{M}$ between $u$ and $v$.  As a result, we have $-y\log(u) + u \leq -y\log(v)+v + (-y/v + 1)\cdot(u-v) + (\mu M)^2\cdot y\cdot(u-v)^2/2$. Choosing $v = y$ gives $-y\log(u) + u \leq -y\log(y)+y + M^2\cdot y\cdot(u-y)^2/2 \leq (\mu M)^{-1}\log (\mu M) + \mu M + (\mu M)^3\cdot(u-y)^2/2$. The result follows by setting $u = \prod_{k=1}^m\theta_{\mathcal{X}_k}$ and taking the expectation of both sides.$\qquad$
\end{proof}

\subsection{Computational Properties}

An equivalent reformulation of (\ref{eqn:nonconvex}) can be defined using the following reparametrization of the loss function
\begin{equation}
R(U) = \mathbb{E}\big(-y\cdot\textstyle\sum_{k = 1}^mu_{\mathcal{X}_k} + \exp\big(\sum_{k=1}^mu_{\mathcal{X}_k}\big)\big), \label{eqn:convex_loss}
\end{equation}
and the relationship between parametrizations is that $u_{\mathcal{X}_k} = \log \theta_{\mathcal{X}_k}$.  The loss function $R(U)$ is convex in $u_{\mathcal{X}_k}$, unlike the original parametrization (\ref{eqn:nonconvex}) which is nonconvex in $\theta_{\mathcal{X}_k}$.  Moreover, the $\prod r_i$ number of constraints in $\Omega$ can be reduced to a polynomial in $\rho$ number of constraints by using a linear program (LP) lift \cite{yannakakis1991}.  Consider the set
\begin{multline}
\Phi = \big\{U : \exists \eta_k,\nu_k \text{ s.t. } \eta_k \leq u_{\mathcal{X}_k} \leq \nu_k,\ -2\log M \leq \eta_k,\ \nu_k \leq 2\log M,\\
-\log M \leq \textstyle\sum_{k = 1}^m\eta_k,\ \textstyle\sum_{k = 1}^m\nu_k\leq \log M,\ \forall x \in \mathcal{R}\big\}.
\end{multline}
We use this to define our reparametrized best approximate hierarchical decomposition as the minimizer to the following convex optimization problem
\begin{equation}
\label{eqn:convex}
\hat{U} = \arg\min \big\{\hat{R}(U)\ \big|\ U \in \Phi\big\},
\end{equation}
where the reparametrized empirical loss function is
\begin{equation}
\hat{R}(U) = \textstyle\frac{1}{n}\sum_{i=1}^n \big(-y\langle i\rangle\cdot\textstyle\sum_{k = 1}^mu_{\mathcal{X}_k\langle i\rangle} + \textstyle\exp\big(\sum_{k = 1}^mu_{\mathcal{X}_k\langle i\rangle}\big)\big).
\end{equation}
The following proposition shows that (\ref{eqn:convex}) is equivalent to (\ref{eqn:nonconvex}).

\begin{proposition}
\label{proposition:convex_equivalence}
Under $\mathbf{A1}$\textrm{--}$\mathbf{A3}$ and for any $\Gamma$, the solution to $(\ref{eqn:convex})$ is equivalent to the solution of $(\ref{eqn:nonconvex})$ with the invertible (under $\mathbf{A1}$) mapping $u_{\mathcal{X}_k} = \log \theta_{\mathcal{X}_k}$.
\end{proposition}

\begin{proof}
We have already argued above that $R(U)$ and $R(\Theta)$ are identical under $u_{\mathcal{X}_k} = \log \theta_{\mathcal{X}_k}$, and so to prove the first part we have to show $\Phi$ is equivalent to $\Omega$ when using the same mapping.  Observe that for points belonging to $\Phi$, we must have $\eta_k \leq \min_{\mathcal{X}_k \in \mathcal{R}_k} u_{\mathcal{X}_k}$ and $\max_{\mathcal{X}_k \in \mathcal{R}_k} u_{\mathcal{X}_k} \leq \nu_k$.  Combining this with the other inequalities defining $\Phi$ leads to $-\log M \leq \min_{x\in\mathcal{R}}\sum_{k = 1}^mu_k$ and $\max_{x\in\mathcal{R}}\sum_{k = 1}^mu_k\leq \log M$, which is the same (under the equivalence) as $M^{-1} \leq \textstyle\prod_{k = 1}^m \theta_{\mathcal{X}_k} \leq M$ from $\Omega$.  A similar argument gives that $-2\log M \leq \eta_k$ and $\nu_k \leq 2\log M$ from $\Phi$ is the same as $M^{-2} \leq \theta_{\mathcal{X}_k} \leq M^2$ from $\Omega$, under the equivalence.  Because the objective and constraints of (\ref{eqn:convex}) and (\ref{eqn:nonconvex}) are the same when equating $u_{\mathcal{X}_k} = \log \theta_{\mathcal{X}_k}$, we have that the solution to (\ref{eqn:convex}) is the same as the solution to (\ref{eqn:nonconvex}).$\qquad$
\end{proof}

We also have the following result about the polynomial-time solvability of (\ref{eqn:nonconvex}).  It is proved by considering the convex reparametrization (\ref{eqn:convex}) and explicitly defining a barrier function for a path-following interior-point method, and then using the methods of \cite{nesterov1994,nemirovski2004} to conduct a complexity analysis for an interior-point method with this barrier function.  The result is stated in terms of the complexity of computing an $\epsilon$-solution, which is a solution $x^\epsilon$ of an optimization problem $f^* = \min \{f(x)\ |\ f_i(x) \leq 0, \forall i;\ x \in G\}$, such that (i) $f(x^\epsilon) - f^* \leq \epsilon$, (ii) $f_i(x^\epsilon) \leq \epsilon$ for all $i$, and (iii) $x^\epsilon \in G$.

\begin{proposition}
\label{proposition:convex}
Under $\mathbf{A1}$--$\mathbf{A3}$ and for any $\Gamma$, an $\epsilon$-solution to the optimization problem $(\ref{eqn:nonconvex})$ can calculated with $O(1)(\rho^3+n^3)\sqrt{\rho+n}\log(\textstyle\frac{1}{\epsilon}(\mu\rho M\log M)(\rho+n))$ arithmetic steps, which is polynomial time in $\mu,\rho,M,n$.
\end{proposition}

\begin{proof}
We suitably modify the proof in \cite{nesterov1994} for the polynomial-time solvability of geometric programs: The first step is to reformulate the convex program (\ref{eqn:convex}) as following the convex program
\begin{align}
\label{eqn:convex_nes}
\min\ &\frac{1}{n}\sum_{i=1}^n \big(-y\langle i\rangle\cdot\textstyle\sum_{k = 1}^mu_{\mathcal{X}_k\langle i\rangle} + t_i\big)\\
\text{s.t. } &\exp\big(\textstyle\sum_{k = 1}^mu_{\mathcal{X}_k\langle i\rangle}\big) -t_i \leq 0,\quad \forall i \in [n]\nonumber\\
&\eta_k - u_{\mathcal{X}_k}\leq 0,\ u_{\mathcal{X}_k} -\nu_k\leq 0,\quad \forall k \in [m], \mathcal{X}_k\in\mathcal{R}_k\nonumber\\
&-\log M -\textstyle\sum_{k = 1}^m\eta_k\leq 0,\ \textstyle\sum_{k = 1}^m\nu_k-\log M\leq 0\nonumber\\
&(u,\eta,\nu, t) \in G \nonumber
\end{align}
where $G = \{(u,\eta,\nu, t) : |\eta_k| \leq 2\log M, |\nu_k| \leq 2\log M, |u_{\mathcal{X}_k}| \leq 2\log M, \forall k \in [m], \mathcal{X}_k\in\mathcal{R}_k; |t_i| \leq M, \forall i \in [n]\}$ is a bounded convex set.  Note $x^+ = (0,0,0,0)$ is the symmetry center of $G$, and so the asymmetry coefficient (see \cite{nesterov1994,renegar2001}) of $G$ with respect to $x^+$ is $\alpha(G : x^+) = 1$.  From Propositions 5.1.3 and 5.4.1 of \cite{nesterov1994}, it follows that
\begin{multline}
\label{eqn:barrier_G}
\textstyle F=-\sum_{k=1}^m\sum_{\mathcal{X}_k\in\mathcal{R}_k}\big\{\log\big(2\log M+u_{\mathcal{X}_k}\big) + \log\big(2\log M-u_{\mathcal{X}_k}\big)\big\}+\\
\textstyle-\sum_{k=1}^m\big\{\log\big(2\log M+\eta_k\big) + \log\big(2\log M-\eta_k\big)+ \\
\log\big(2\log M+\nu_k\big) +\log\big(2\log M-\nu_k\big)\big\}+\\
\textstyle -\sum_{i=1}^n\big\{\log\big(M+t_i\big) + \log\big(M-t_i\big)\}
\end{multline}
is a $(2\rho+4m+2n)$-self-concordant barrier for $G$.

The next step is to bound the objective and constraints of (\ref{eqn:convex_nes}).  Note $\textbf{A1},\textbf{A2}$ imply the absolute value of the objective is upper bounded by $\mu mM\log M + M$.  Similarly, $\textbf{A1},\textbf{A2}$ imply the absolute value of the left hand side of the constraints are upper bounded by $2M$, $4\log M$, $4\log M$, $\log M + 2m\log M$, $\log M + 2m\log M$, respectively.  Consequently, an upper bound on these upper bounds is $V = 4\mu mM\log M + 2M$.

The third step is to identify barrier functions for the epigraphs of each constraint in (\ref{eqn:convex_nes}).  Proposition 5.4.1 of \cite{nesterov1994} states $-\log(-\tau)$ is a 1-self-concordant barrier for the constraint $\tau \leq 0$.  Similarly, Proposition 5.3.3 of \cite{nesterov1994} states $-\log(\log\big(\tau)-\xi) - \log(\tau)$ is a 2-self-concordant barrier for the constraint $\exp(\xi) \leq \tau$.  Consequently, the following
\begin{multline}
\label{eqn:barrier}
\textstyle-\log\big(3Vt/\epsilon - V - \frac{1}{n}\sum_{i=1}^n \big(-y\langle i\rangle\cdot\textstyle\sum_{k = 1}^mu_{\mathcal{X}_k\langle i\rangle} + t_i\big)\big) + \\
\textstyle-\sum_{i=1}^n\big\{\log\big(\log(t_i+t)-\sum_{k = 1}^mu_{\mathcal{X}_k\langle i\rangle}\big) +\log\big(t_i+t\big)\big\} + \\
\textstyle-\sum_{k=1}^m\sum_{\mathcal{X}_k\in\mathcal{R}_k}\big\{\log\big(-\eta_k+u_{\mathcal{X}_k}\big)+\log\big(-u_{\mathcal{X}_k}+\nu_k\big)\big\}+\\
\textstyle-\log\big(\log M+\sum_{k=1}^m\eta_k\big)-\log\big(\log M-\sum_{k=1}^m\nu_k\big)+F
\end{multline}
is a $(3\rho+4m+4n+3)$-self-concordant barrier function by Proposition 5.1.3 of \cite{nesterov1994}.  Note we have the bound $\rho \geq m$ from the definition of $\rho$.  From the results of \S 6.1 of \cite{nesterov1994}, an $\epsilon$-solution to (\ref{eqn:convex}) can be found in $O(1)\sqrt{\rho+n}\log(\textstyle\frac{1}{\epsilon}(\mu\rho M\log M)(\rho+n))$ steps of the path-following algorithm.  The Newton system for one step is assembled in $O(nm^2+\rho)$ arithmetic steps and can be solved in $O(\rho^3+n^3)$ arithmetic steps.  Consequently, an $\epsilon$-solution to (\ref{eqn:convex}) can be found in $O(1)(\rho^3+n^3)\sqrt{\rho+n}\log(\textstyle\frac{1}{\epsilon}(\mu\rho M\log M)(\rho+n))$ arithmetic steps.  The proof concludes by noting Proposition \ref{proposition:convex_equivalence} implies an $\epsilon$-solution to (\ref{eqn:nonconvex}) can be calculated by applying the transformation $\theta_{\mathcal{X}_k} = \exp(u_{\mathcal{X}_k})$ to the $\epsilon$-solution to (\ref{eqn:convex}).$\qquad$
\end{proof}

This result immediately implies that the best rank-1 approximation of a positive tensor can be computed in polynomial time, which is in contrast to the general case where computing the best rank-1 approximation is NP-hard \cite{hillar2013}.  The approximation problem becomes easier when we restrict our focus to positive tensors.

\begin{corollary}
The best rank-1 approximation, under the loss function $(\ref{eqn:nonconvex_risk})$ and satisfying $\mathbf{A1}$, of a tensor $\psi$ can be computed in polynomial time with a number of arithmetic steps that is polynomial in $r,p,n,\mu,M$.
\end{corollary}

\begin{proof}
The best rank-1 approximation corresponds to a partition decomposition with $\text{facets}(\Gamma) = \{\{1\},\ldots,\{p\}\}$, and so the result follows from Proposition \ref{proposition:convex}.$\qquad$
\end{proof}

\subsection{Bound on Squared Error in Terms of Loss Function}

Define the oracle parameters to be any $\Theta^* \in \arg\min \{R(\Theta)\ |\ \Theta \in \Omega\}$.  Below, we provide a relationship between a squared error function involving $\Theta^*$ and the loss function (\ref{eqn:nonconvex}).  This relationship will serve as useful machinery for proving subsequent results.

\begin{proposition}
\label{prop:weak_equivalence}
Under $\mathbf{A1}$,$\mathbf{A2}$ and for any $\Gamma$, we have for any $\Theta \in \Omega$ that $\textstyle \frac{1}{2M^3}\cdot \textstyle\mathbb{E}((\prod_{k=1}^m\theta_{\mathcal{X}_k} - \prod_{k=1}^m\theta^*_{\mathcal{X}_k})^2) \leq R(\Theta) - R(\Theta^*)$.
\end{proposition}

\begin{proof}
We will use the equivalent (by Proposition \ref{proposition:convex_equivalence}) convex reparameterization in $U$ to show the necessary bound.  The first-order optimality condition \cite{rockafellar2009} for the reparametrized optimization problem (\ref{eqn:convex}) is
\begin{equation}
D(U^*,U) = \nabla R(U^*)\cdot(U - U^*) \geq 0, \label{eqn:fooc}
\end{equation}
for all $U \in \Phi$.  Since the probability space of $x \in \mathcal{R}$ is finite, we can interchange the order of differentiation and integration as shown below
\begin{align}
\partial_{\mathcal{X}_j} R(U) &= \textstyle\sum_{x \in \mathcal{R}}f_x\cdot\big(\partial_{\mathcal{X}_j}\big(-\psi_x\cdot\textstyle\sum_{k = 1}^mu_{\mathcal{X}_k} + \exp\big(\sum_{k=1}^mu_{\mathcal{X}_k}\big)\big)\big)\\
&= \textstyle\sum_{x \in \mathcal{R}}f_x\cdot\big(-\psi_x + \exp\big(\textstyle\sum_{k=1}^mu_{\mathcal{X}_k}\big)\big)\cdot\mathbbm{1}_{\mathcal{X}_j}, \label{eqn:fubini}
\end{align}
where $f_x = \mathbb{P}(x)$.  Combining (\ref{eqn:fooc}) and (\ref{eqn:fubini}) leads to 
\begin{align}
D(U^*,U) &= \textstyle\sum_{\mathcal{X}_j}\sum_{x \in \mathcal{R}}f_x\cdot\big(-\psi_x + \exp\big(\textstyle\sum_{k=1}^mu^*_{\mathcal{X}_k}\big)\big)\cdot\mathbbm{1}_{\mathcal{X}_j}\cdot\big(u^{}_{\mathcal{X}_j} - u^{*}_{\mathcal{X}_j}\big)\\
&=\textstyle\mathbb{E}\big(\big(-\psi_x + \exp\big(\textstyle\sum_{k=1}^mu^*_{\mathcal{X}_k}\big)\big)\cdot\big(\textstyle\sum_{k=1}^mu^{}_{\mathcal{X}_k}-\sum_{k=1}^mu^*_{\mathcal{X}_k}\big)\big). \label{eqn:fooc_sum}
\end{align}

Next, consider $f(u) = -yu + e^u$.  Since $f''(u) \geq e^a$ for all $u \in [a,b]$, this function is strongly convex \cite{boyd2004} and satisfies $-yu+e^u \geq -yv+e^v + (-y + e^v)\cdot(u-v) + e^a/2\cdot(u-v)^2$, for all $u,v \in [a,b]$.  Applying this inequality to $R(U)$ gives that for any $U \in \Phi$, $R(U) \geq R(U^*) + D(U^*,U) + \textstyle \frac{1}{2M}\cdot\mathbb{E}((\sum_{k=1}^mu^*_{\mathcal{X}_k} - \sum_{k=1}^mu_{\mathcal{X}_k})^2)$, where we have used (\ref{eqn:fooc_sum}) to simplify the expression.  Since $D(U^*,U) \geq 0$ from (\ref{eqn:fooc}), we have that for any $U \in \Phi$, $R(U) - R(U^*) \geq \textstyle \frac{1}{2M}\cdot\mathbb{E}((\sum_{k=1}^mu^*_{\mathcal{X}_k} - \sum_{k=1}^mu_{\mathcal{X}_k})^2)$. Because $e^u$ is Lipschitz on bounded domains (i.e., $|e^u - e^v| \leq e^{\log M}\cdot|u - v|$, for all $-\log M \leq u,v \leq \log M$), we have that for any $U \in \Phi$, $R(U) - R(U^*) \geq \textstyle \frac{1}{2M^3}\cdot\mathbb{E}((\exp(\sum_{k=1}^mu^*_{\mathcal{X}_k}) - \exp(\sum_{k=1}^mu_{\mathcal{X}_k}))^2)$. Inverting the mapping $u_{\mathcal{X}_k} = \log \theta_{\mathcal{X}_k}$, which is possible because of \textbf{A1}, gives $R(\Theta) - R(\Theta^*) \geq \textstyle \frac{1}{2M^3}\cdot\mathbb{E}((\prod_{k=1}^m\theta^*_{\mathcal{X}_k} - \prod_{k=1}^m\theta_{\mathcal{X}_k})^2)$.$\qquad$
\end{proof}

\subsection{Risk Consistency via Rademacher Complexity}

Having shown that the loss function (\ref{eqn:nonconvex_risk}) has promising properties, we next identify sufficient conditions for risk consistency \cite{bartlett2002,greenshtein2004,kakade2009,lafferty2012}.  Our approach is to interpret the problem as a high-dimensional (though lower-dimensional than if we had not taken the low-rank tensor structure into consideration) linear regression under a Lipschitz loss function.  The linear regression will not be with respect to the indices $x$, but will instead be defined using indicator functions.  With this interpretation, we will use Rademacher averages \cite{bartlett2002,kakade2009,lafferty2012,boucheron2013} to bound the complexity of our model (\ref{eqn:nonconvex_risk}).  

\begin{proposition}
\label{prop:concentration}
Under $\mathbf{A1}$--$\mathbf{A3}$ and for any $\Gamma$, we have
\begin{equation}
\mathbb{P}\Big(\sup_{\Theta \in \Omega}\big|\hat{R}(\Theta) - R(\Theta)\big| < t\Big) \geq 1 - \exp\Big(-C_1n\Big(t-C_2\sqrt{\textstyle\frac{m\rho}{n}}\Big)^2\Big),
\end{equation}
where constants $C_1,C_2 > 0$ depend on $\mu,M$.
\end{proposition}

\begin{proof}
The proof proceeds similarly to \cite{bartlett2002,lafferty2012} by bounding the deviation of the supremum from the expectation of the supremum, and it will be easier to work in the reparametrized space.  First, note that $\hat{R}(U)$ satisfies the bounded deviation condition with constant $(\mu M\log M + M)/n$ \cite{boucheron2013} because of \textbf{A1},\textbf{A2}.  As a result, McDiarmid's inequality \cite{boucheron2013} gives
\begin{equation}
\label{eqn:mcdiarmids}
\mathbb{P}\Big(\sup_{U \in \Phi} \big|\Delta(U)\big| - \mathbb{E}\big(\sup_{U \in \Phi}\big|\Delta(U)\big|\big) > t\Big) \leq \exp\Big(\textstyle\frac{-2nt^2}{(\mu M\log M + M)^2}\Big),
\end{equation}
where $\Delta(U) = \hat{R}(U) - R(U)$.  And so the result follows if we can bound the quantity $\mathbb{E}(\sup_{U \in \Phi}|\Delta(U)|)$.  Because the loss function $\phi(z) = -yz + e^z$ (for a fixed value of $y$ and for $z \in [-\log M,\log M]$) is Lipschitz with respect to $z$ with Lipschitz constant $L = \mu M + M$, structural results \cite{ledoux1991,bartlett2002} give that 
\begin{equation}
\label{eqn:radup}
\mathbb{E}\Big(\sup_{U \in \Phi}\big|\Delta(U)\big|\Big) \leq 4L\cdot\mathsf{R}(\mathsf{F}_\mathsf{W}),
\end{equation}
where $\mathsf{R}(\mathsf{F}_\mathsf{W})$ is the Rademacher complexity for an appropriate linear function class.  In particular, we can define our empirical loss by taking the sample average of $\phi$ composed with the linear model $\textstyle\sum_{k=1}^m\sum_{\mathcal{X}_k \in \mathcal{R}_k} \mathbbm{1}_{\mathcal{X}_k = \mathcal{X}_k\langle i\rangle}\cdot u_{\mathcal{X}_k}$. We should interpret the terms $\mathbbm{1}_{\mathcal{X}_k = \mathcal{X}_k\langle i\rangle}$ as pseudo-predictors, and the $u_{\mathcal{X}_k}$ are still the parameters.  The key observation is that if we define $\chi \in \{0,1\}^\rho$ to be the vector of pseudo-predictors, then in fact $\|\chi\|_1 = m$, $\|\chi\|_2 = \sqrt{m}$, and $\|\chi\|_\infty = 1$.  Recall that $\Phi$ is defined so that $\|u_{\mathcal{X}_k}\|_\infty \leq 2\log M$.  And so results from \cite{kakade2009} imply that $\mathsf{R}(\mathsf{F}_\mathsf{W}) \leq 2\log M\cdot\sqrt{m\rho/n}$. The result follows by combining this with (\ref{eqn:mcdiarmids}) and (\ref{eqn:radup}).$\qquad$
\end{proof}

The above result can be used to show risk consistency of the $\epsilon$-solution $\hat{\Theta}^\epsilon$ to the best approximate hierarchical decomposition problem (\ref{eqn:nonconvex}):

\begin{theorem}
\label{thm:est_consistent}
Under $\mathbf{A1}$--$\mathbf{A3}$ and for any $\Gamma$, with probability at least $1 - c_1n^{-1}$ we have $0 \leq R(\hat{\Theta}^\epsilon) - R(\Theta^*) \leq \sqrt{m \rho\log n/n} + \epsilon$, where constant $c_1 > 0$ depends on $\mu,M$.
\end{theorem}

\begin{proof}
The proof follows that in \cite{greenshtein2004} with modifications to extend the argument for $\epsilon$-solutions.  The triangle inequality implies $|R(\hat{\Theta}^\epsilon) - R(\Theta^*)| \leq |R(\hat{\Theta}^\epsilon) - \hat{R}(\hat{\Theta}^\epsilon)| + |R(\Theta^*) - \hat{R}(\hat{\Theta}^\epsilon)|$, and so we need to bound these two terms.  The first term $|R(\hat{\Theta}^\epsilon) - \hat{R}(\hat{\Theta}^\epsilon)|$ is bounded by Proposition \ref{prop:concentration}, and so we only need to focus on the second term $|R(\Theta^*) - \hat{R}(\hat{\Theta}^\epsilon)|$.  Because the quantity $\hat{\Theta}^\epsilon$ is an $\epsilon$-solution to an optimization problem with objective function $\hat{R}(\cdot)$, we have $\hat{R}(\hat{\Theta}^\epsilon) \leq \hat{R}(\Theta^*) + \epsilon \Rightarrow \hat{R}(\hat{\Theta}^\epsilon) - R(\Theta^*) \leq \hat{R}(\Theta^*) - R(\Theta^*) + \epsilon$.  Similarly, because $\Theta^*$ is the minimizer of $R(\cdot)$, we have $R(\Theta^*) \leq R(\hat{\Theta}^\epsilon) \Rightarrow \hat{R}(\hat{\Theta}^\epsilon) - R(\Theta^*) \geq \hat{R}(\hat{\Theta}^\epsilon) - R(\hat{\Theta}^\epsilon)$.  The result follows from combining the above with Proposition \ref{prop:concentration}.$\qquad$
\end{proof}

\subsection{Polynomial-Time Hierarchical Decompositions and Approximations}

We are now in a position to provide a randomized algorithm that can compute a hierarchical decomposition with time that is polynomial in $\rho(\Gamma)$, as opposed to the algorithm given in the proof of Proposition \ref{proposition:well-posed} that is polynomial in $\prod r_i$.  Let $\delta$ be a parameter that controls the approximation accuracy of the decomposition.  Then given any $\Gamma$, the algorithm is as follows:
\begin{enumerate}
\item Set $n = \rho(\Gamma)/\delta$;
\item Sample indices $x\langle i\rangle \in\mathcal{R}$ and record the corresponding tensor entries $y\langle i\rangle$, for $i=1,\ldots,n$;
\item Compute $\hat{U}^\epsilon$ by solving (\ref{eqn:nonconvex});
\item Compute $\hat{\Theta}^\epsilon$ by using the mapping $u_{\mathcal{X}_k} = \log \theta_{\mathcal{X}_k}$.
\end{enumerate}

We will use Proposition \ref{prop:weak_equivalence} and Theorem \ref{thm:est_consistent} to reason about errors measured using the squared loss function.  Recall that $\delta$ controls the approximation accuracy of the decomposition, and the value $\epsilon$ controls the accuracy of the optimization solution.

\begin{theorem}
\label{thm:ranres}
Suppose $\mathbf{A1}$--$\mathbf{A3}$ hold, $\Gamma$ is correct for $\psi_x$, and the indices $x\langle i\rangle$ are sampled uniformly from $\mathcal{R}$.  Then with probability at least $1 - c_1\delta/\rho$ the above algorithm computes a hierarchical decomposition $\hat{\Theta}^\epsilon$ with average approximation error
\begin{equation}
\label{eqn:approx_error}
\textstyle(\prod r_i)^{-1}\sum_{x\in\mathcal{X}}\big(\big(\psi_x - \prod_{k=1}^m\hat{\theta}^\epsilon_{\mathcal{X}_k}\big)^2\big) \leq 2M^3\big(\sqrt{m \delta\log (\rho/\delta)}+\epsilon\big)
\end{equation}
and has a polynomial-time arithmetic cost $O(1)(\rho/\delta)^{7/2}\log\big(\textstyle\frac{1}{\delta\epsilon}(\mu\rho^2 M\log M)\big)$, where constant $c_1 > 0$ depends on $\mu,M$.
\end{theorem}

\begin{proof}
We can replace $\prod_{k=1}^m\theta^*_{\mathcal{X}_k}$ with $\psi_x$ because of Proposition \ref{prop:oracle_equiv}.  Combining Theorem \ref{thm:est_consistent} with Proposition \ref{prop:weak_equivalence} gives $\frac{1}{2M^3}\cdot \textstyle\mathbb{E}((\prod_{k=1}^m\hat{\theta}_{\mathcal{X}_k}^\epsilon - \psi_x)^2)\leq\sqrt{m \delta\log (\rho/\delta)}+\epsilon$ with probability at least $1 - c_1\delta/\rho$.  When $x\langle i\rangle$ is uniformly sampled, this expectation can be written as $\mathbb{E}((\psi_x - \textstyle\prod_{k=1}^m\hat{\theta}^\epsilon_{\mathcal{X}_k})^2) = (\prod r_i)^{-1}\sum_{x\in\mathcal{X}}\textstyle((\psi_x - \prod_{k=1}^m\hat{\theta}^\epsilon_{\mathcal{X}_k})^2)$.  Proposition \ref{proposition:convex} states $\hat{\Theta}^\epsilon$ has arithmetic cost $O(1)(\rho/\delta)^{7/2}\log(\textstyle\frac{1}{\epsilon}(\mu\rho^2 M\log M/\delta))$.$\qquad$
\end{proof}

When $\Gamma$ is incorrect for $\psi_x$, only a weaker result is possible because it will be the case that $\psi_x \neq \prod_{k=1}^m\theta^*_{\mathcal{X}_k}$ for the oracle parameters $\Theta^*$.  In other words, the previous result states we can compute hierarchical decompositions (with a correct $\Gamma$) in polynomial time using an amount of data that depends on effective dimension $\rho$ (rather than the number of tensor entries $\prod r_i$), whereas the following result states we can compute best tensor approximations (where the approximation has a hierarchical decomposition given by a potentially incorrect $\Gamma$) in polynomial time using an amount of data that depends on effective dimension $\rho$.

\begin{theorem}
\label{thm:ranres2}
Suppose $\mathbf{A1}$--$\mathbf{A3}$ hold, and the indices $x\langle i\rangle$ are sampled uniformly from $\mathcal{R}$.  Then with probability at least $1 - c_1\delta/\rho$ the above algorithm computes an approximate hierarchical decomposition $\hat{\Theta}^\epsilon$ with average approximation error
\begin{equation}
\textstyle(\prod r_i)^{-1}\sum_{x\in\mathcal{X}}\textstyle\big(\big(\prod_{k=1}^m\theta^*_{\mathcal{X}_k} - \prod_{k=1}^m\hat{\theta}^\epsilon_{\mathcal{X}_k}\big)^2\big) \leq 2M^3\big(\sqrt{m \delta\log (\rho/\delta)}+\epsilon\big)
\end{equation}
and has a polynomial-time arithmetic cost $O(1)(\rho/\delta)^{7/2}\log\big(\textstyle\frac{1}{\delta\epsilon}(\mu\rho^2 M\log M)\big)$, where constant $c_1 > 0$ depends on $\mu,M$.
\end{theorem}

\begin{proof}
This result is proved in the proof of Theorem \ref{thm:ranres}.$\qquad$
\end{proof}

One of the implications of this is that we can compute the best (as measured by the loss (\ref{eqn:nonconvex_risk})) approximate hierarchical decomposition of $\psi$ in an amount of time that is polynomial in the effective dimension $\rho$ induced by the simplicial complex $\Gamma$.  We can also specialize these results to rank-1 approximations.

\begin{corollary}
The best rank-1 approximation, under the loss function $(\ref{eqn:nonconvex_risk})$ and satisfying $\mathbf{A1}$, of a tensor $\psi$ can be computed in polynomial time with a number of arithmetic steps that is polynomial in $r,p,\mu,M$.
\end{corollary}

\begin{proof}
The best rank-1 approximation corresponds to a partition decomposition with $\text{facets}(\Gamma) = \{\{1\},\ldots,\{p\}\}$, and so the result follows from Theorem \ref{thm:ranres2}.$\qquad$
\end{proof}

\section{Tensor Completion}

\label{section:tc}

The tensor completion problem is almost the same as computing a hierarchical decomposition using data samples of tensor entries, and the only difference is we must also determine $\Gamma$ using the measured data before computing the tensor approximation.  A key assumption for formulating completion problems is the object being estimated has low rank, because this allows for statistically consistent estimation with a small number of measurements. The usual idea for solving the completion problem is to formulate an optimization problem in which the rank (or a rank surrogate) of the estimated object is minimized subject to the entries that have been observed being equal to the corresponding entries in the estimated object.

Matrix completion is well-studied \cite{fazel2001,zou2006,recht2010,cai2010,candes2011,ma2011,agarwal2012,saunderson2012,keshavan2010,chatterjee2014,gavish2014}, but tensor completion is still an open problem.  Tensor rank is NP-hard to compute \cite{hillar2013} and has poor continuity properties \cite{desilva2008}.  As a result, existing approaches use multilinear rank as a surrogate for tensor rank because it can be computed in polynomial time \cite{mu2014} and has better continuity properties than tensor rank \cite{desilva2008}.  The multilinear (or Tucker) rank of tensor $\psi$ is the vector: $\text{rank}_\boxplus(\psi) = (\text{rank}(\psi_{(1)}), \ldots,\text{rank}(\psi_{(p)}))$, where $\text{rank}(\cdot)$ is the standard matrix rank and $\psi_{(k)}$ is the unfolding of the tensor (into a matrix) along the $k$-th index \cite{desilva2008,mu2014}.  

Several approaches to tensor completion \cite{tomioka2010,signoretto2010,gandy2011,liu2013,mu2014,zhang2014,montanari2014} use soft-thresholding on the multilinear rank, because this converts the problem into the well-studied matrix completion problem.  A canonical formulation \cite{tomioka2010,gandy2011} is to solve $\min_{\hat{\psi}} \frac{1}{n}\sum_{i=1}^n (y\langle i\rangle - \hat{\psi}_{x\langle i\rangle})^2 + \sum_{k=1}^p\lambda_k\cdot\|\hat{\psi}_{(k)}\|_*$, where $\|\cdot\|_*$ denotes the nuclear norm and $\lambda_k > 0$ are weightings.  There are open questions on the optimal weighting $\lambda_k$ in this formulation \cite{oymak2012,mu2014,zhang2014}.  Another class of approaches use power iteration, message passing, or alternating minimization algorithms \cite{goldfarb2014,montanari2014}; these are local approaches that are not guaranteed to provide statistical consistency in general, though they can empirically work well on specific problem instances.

However, there is a large gap with the statistical convergence rates achievable by the above methods \cite{tomioka2010,signoretto2010,gandy2011,liu2013,mu2014,zhang2014,montanari2014}.  If we define $\pi = \max_j\text{rank}_\boxplus(\psi)$, then existing convex optimization-based algorithms need $O(\pi^{\lfloor p/2\rfloor}r^{\lceil p/2\rceil})$ measurements; this is substantially worse than the best rates achievable by an NP-hard formulation, which needs $O(\pi^p + r\pi p)$ measurements \cite{mu2014}.  Another analysis focusing on rank-1 tensors found that existing local approaches like power iteration and message passing need a diverging signal-to-noise ratio to achieve statistical consistency \cite{montanari2014}.

Here, we use effective dimension $\rho(\Gamma)$ of a partition $\Gamma$ as a surrogate for tensor rank.  As discussed in \S \ref{section:erc}, $\rho$ majorizes tensor rank, and so a small $\rho$ corresponds to a low tensor rank.  We will focus on the case where $\Gamma$ is a partition and leave open the more general case of a $\Gamma$ that is a general simplicial complex.  One advantage of developing tensor completion algorithms using effective dimension is that we will be able to show this leads to methods that achieve statistical consistency using only slightly more measurements than NP-hard formulations in specific cases.

The section begins by discussing which correct partition would be best for statistical purposes.  Next, we define a test statistic that can distinguish whether two indices belong to the same or different facets of this partition.  This statistic can be used to construct a partition $\Gamma$ from the data, and we consequently use this test statistic to construct a polynomial-time algorithm for tensor completion.  A theoretical analysis shows our algorithm needs exponentially less data for specific cases than existing tensor completion methods \cite{tomioka2010,signoretto2010,gandy2011,liu2013,mu2014,zhang2014,montanari2014}.  The section concludes by discussing how to modify the algorithm to handle a tradeoff between purposely choosing an incorrect partition (which increases statistical bias) in exchange for significant reduction in the effective dimension (which reduces statistical variance).

Before we begin, we make a note regarding the computational complexity of computing estimates by solving optimization problems.  It is typical in the statistics literature \cite{fazel2001,zou2006,recht2010,cai2010,candes2011,ma2011,agarwal2012,saunderson2012,keshavan2010,chatterjee2014,gavish2014,tomioka2010,signoretto2010,gandy2011,liu2013,mu2014,zhang2014,montanari2014} to not make a distinction between $\epsilon$-solutions and exact solutions to optimization problems.  The reason is that the $\epsilon$-solutions generally add only an $\epsilon$ term in upper bounds on error (see for instance Theorems \ref{thm:ranres} and \ref{thm:ranres2}).  In this section, we follow this convention from statistics and do not distinguish between $\epsilon$-solutions and exact solutions.  Polynomial time computation is ensured in our case because of Proposition \ref{proposition:convex}.

\subsection{Specifying the Ideal Partition}

Despite the lack of uniqueness of correct $\Gamma$ for a tensor $\psi_x$, some correct $\Gamma$ are better than others.  Statistically, a correct $\Gamma$ with the smallest effective dimension $\rho(\Gamma)$ is the best choice because this reduces the number of parameters to estimate and leads to more efficient methods.  As a result, we define the \emph{ideal partition} $\Gamma^*$ to be a partition such that $\rho(\Gamma^*) \leq \rho(\Gamma)$ for all correct partitions $\Gamma$.  An ideal partition must always exist because $\text{facets}(\Gamma) = \{\{1,\ldots,p\}\}$ is a correct partition, but an ideal partition may not always have low effective dimension: In a subsequent subsection, we will discuss low-rank approximations that can be used for this situation.  However, if an ideal partition exists then it must be unique because otherwise we could use all ideal partitions to define a new correct partition $\Gamma$ with a strictly smaller effective dimension $\rho(\Gamma)$.  Additionally, because the loss function (\ref{eqn:nonconvex_risk}) is a Bregman divergence, it is known \cite{banerjee2005} that the minimal possible risk is
\begin{equation}
\label{eqn:bregman}
R(\psi_x) := \arg\min \big\{R(\Theta)\ \big|\ \Theta \in \Omega,\ \forall \text{ choices of } \Gamma\big\} = R(\Theta^*(\Gamma^*)),
\end{equation}
where $R(\Theta^*(\Gamma^*))$ denotes the minimum risk under the ideal partition $\Gamma^*$.

\subsection{Defining the Risk Gap}

We begin by making a minor assumption about the distribution of the predictors:\\

\noindent \textbf{A4}. The $x\langle i\rangle$ are iid random variables with distribution such that $x_u\langle i\rangle$ is independent of $x_v\langle i\rangle$ whenever $u \in F_j$ and $v \in F_k$, where $F_j,F_k \in \text{facets}(\Gamma^*)$ and $F_j\neq F_k$.\\

\noindent This independence assumption is similar to assumptions typically made for low-rank matrix and tensor completion (e.g., \cite{tomioka2010,chatterjee2014,mu2014}).  The typical assumption is that entries of the matrix (or tensor) are sampled with uniform probability, which is equivalent to assuming the $x_u\langle i\rangle$ are jointly independent \cite{landsberg2012}.  Here, we only require independence between indices that belong to different facets of $\Gamma^*$.  It is useful to emphasize that uniform sampling of entries would satisfy our assumption.

This assumption could be generalized to include approximately independent distributions.  For instance, consider the distribution on $x$ given by: $f_x = \textstyle(1-\epsilon)\cdot\bigotimes_{k=1}^p V_k + \epsilon \cdot g_x$, where $V_k \in \mathbb{R}^{r_k}$ are nonnegative vectors that sum to one $\sum_j V_k^j = 1$, $\epsilon \ll 1$ is a small constant, and $g_x$ is an arbitrary probability distribution on $x$.  Because $\bigotimes_{k=1}^p V_k$ represents a distribution where each $x_k$ is jointly independent \cite{landsberg2012}, we can interpret the distribution $f_x$ as having approximate independence between the $x_k$.  Under such conditions, we can bound the error incurred by our estimators assuming \textbf{A4}.  We do not consider the details of this generalization in this paper.

Our main idea is that we can determine structural properties of the ideal partition $\Gamma^*$ by computing quantities that are statistically easy to estimate.  In particular, we define the \emph{risk gap} of two indices $j,q$ to be the test statistic $\mathcal{G}_{jq} = \min\{\overline{R}_{jq}(\overline{B})\ |\ \overline{B} \in \overline{\Phi}\} - \min\{R_{jq}(B)\ |\ B \in \Phi\}$, where $\overline{\Phi}$ is the set (\ref{eqn:nonconvex_constraint}) for the partition $\{\{j\},\{q\}\}$, $\Phi$ in this case is the set (\ref{eqn:nonconvex_constraint}) for the partition $\{\{j,q\}\}$, $R_{jq}(B) = \mathbb{E}(-y\log (\beta_{x_j,x_q}) + \beta_{x_j,x_q})$, and $\overline{R}_{jq}(\overline{B}) = \mathbb{E}(-y\cdot\textstyle(\log \overline{\beta}_{x_j} + \log\overline{\beta}_{x_q}) + \overline{\beta}_{x_j}\overline{\beta}_{x_q})$.  Note that $\mathcal{G}_{jq} \geq 0$ because
\begin{equation}
\label{eqn:2dreform}
\min\big\{\overline{R}_{jq}(\overline{B})\ \big|\ \overline{B} \in \overline{\Phi}\big\} = \min\big\{R_{jq}(B)\ \big|\ B \in \Phi,\ \beta_{x_j,x_q} = \overline{\beta}_{x_j}\overline{\beta}_{x_q},\ \overline{B} \in \overline{\Phi}\}.
\end{equation}
The following structural characterization is essentially a corollary to results in \cite{banerjee2005}.

\begin{proposition}
\label{proposition:rank_one_risk}
Suppose $\mathbf{A1}$, $\mathbf{A2}$, $\mathbf{A4}$ hold and that $\Gamma^*$ is an ideal partition.  If indices $j,q$ are such that there is no $F_k\in\mathrm{facets}(\Gamma^*)$ with $j,q \in F_k$, then $\mathcal{G}_{jq} = 0$.
\end{proposition}

\begin{proof}
Because of \textbf{A4}, without loss of generality we have $F_1,F_2\in\text{facets}(\Gamma^*)$ such that $j \in F_1$, $q \in F_2$, and $F_1 \neq F_2$.  Next observe that $\beta_{x_j,x_q} = \mathbb{E}[y | x_j, x_q]$ minimizes $\min\{R_{jq}(B)\ |\ B \in \Phi\}$ \cite{banerjee2005}; this conditional expectation can be written as $\beta_{x_j,x_q} = \textstyle\mathbb{E}[\prod_{k = 1}^m\theta_{\mathcal{X}_k}|x_j,x_q] = \mathbb{E}[\theta_{\mathcal{X}_1}|x_j]\cdot\mathbb{E}[\theta_{\mathcal{X}_2}|x_q]\cdot\mathbb{E}[\prod_{k = 3}^m\theta_{\mathcal{X}_k}]$. Defining $v_{x_j} = \mathbb{E}[\theta_{\mathcal{X}_1}|x_j]$, $w_{x_q} = \mathbb{E}[\theta_{\mathcal{X}_2}|x_q]$, and $\kappa = \mathbb{E}[\prod_{k = 3}^m\theta_{\mathcal{X}_k}]$, we can write the conditional expectation as $\beta_{x_j,x_q} = \kappa\cdot v_{x_j} w_{x_q}$.  Because \textbf{A1} holds and $\Gamma^*$ is an ideal parition, this means the $\theta_{\mathcal{X}_k}$ are strictly positive.  As a result, we have that (i) $\kappa$ is a strictly positive constant, and (ii) the vectors $u_{x_j},v_{x_q}$ have strictly positive entries.  

Next, observe that if we can choose $\overline{\beta}_{x_j}$ and $\overline{\beta}_{x_q}$ such that $\beta_{x_j,x_q} = \overline{\beta}_{x_j}\overline{\beta}_{x_q}$ and $\overline{B} \in \overline{\Phi}$, then the result follows because the minimizer to $\min\{R_{jq}(B)\ |\ B \in \Phi\}$ also gives the minimizer to (\ref{eqn:2dreform}).  In fact, such a choice is guaranteed to exist by Proposition \ref{proposition:well-posed} applied to $\beta_{x_j,x_q}$, since we have the decomposition $\beta_{x_j,x_q} = \kappa\cdot u_{x_j} v_{x_q}$.$\qquad$
\end{proof}

This is a useful result because it says that important structural information is encoded in an object $\mathcal{G}_{jq}$ that is easy to estimate.  Unfortunately, the converse is not true.  Consider the counterexample with $p = 3$, $r_1=r_2=r_3=2$, and
\begin{equation}
\label{eqn:tensor_ex}
\begin{pmatrix}[c|c] \psi_{x_1,x_2,1} & \psi_{x_1,x_2,2} \end{pmatrix} = \begin{pmatrix}[cc|cc] 2 & 0 & 0 & 2\\ 0 & 2 & 2 & 0 \end{pmatrix},
\end{equation}
where the entries are measured uniformly.  It can be shown that $\mathcal{G}_{jq} = 0$, but the hypothesis of Proposition \ref{proposition:rank_one_risk} does not hold.  Consequently, we will have to restrict the class of low-rank tensors we consider by defining an incoherence condition:\\

\noindent \textbf{A5}. There exists $\alpha > 0$ such that $\mathcal{G}_{jq} \geq \alpha$ for all $j,q \in F_k$ and all $F_k \in \text{facets}(\Gamma^*)$.\\

Incoherence conditions are common in the matrix and tensor completion literature \cite{recht2010,cai2010,candes2011,ma2011,agarwal2012,saunderson2012,keshavan2010,chatterjee2014,gavish2014}.  One interpretation of \textbf{A5} is it forces $\mathcal{G}_{jq}$ to represent the difference in risk (flattened to just two variables) between keeping $x_j,x_q$ coupled versus decoupled in the ideal partition $\Gamma^*$.  The existence of tensors satisfying this condition can be seen by considering the second example from \S \ref{section:pd} in which $\psi_x = \prod_{k = 1}^m\theta_{\mathcal{X}_k} = (\theta^{(1)}\otimes\cdots\otimes\theta^{(m)})_x$, where $\theta^{(1)},\ldots,\theta^{(s)}$ are matrices and $\theta^{(s+1)},\ldots,\theta^{(m)}$ are vectors.  Assume that (i) the entries of $\theta^{(k)}$ lie within the set $[M^{-1/m},M^{1/m}]$, for all $k =1,\ldots,m$, (ii) we sample uniformly from the tensor, and (iii) there is a constant $\alpha > 0$ such that the singular values of each matrix (i.e., $k = 1,\ldots,s$) satisfy $\textstyle\sum_{\gamma \geq 2}\sigma_\gamma(\theta^{(k)}) \geq \sqrt{2}(rM)^{3/2}\sqrt{\alpha}$, where $\sigma_\gamma(\cdot)$ are the singular values sorted into decreasing order.  Then Proposition \ref{prop:equivalence} gives $\mathcal{G}_{jq} \geq \min \{\textstyle\frac{1}{2M^3}\cdot \mathbb{E}((\overline{\beta}_{x_j}\overline{\beta}_{x_q} - \mathbb{E}[y|x_j,x_q])^2)\ |\ B \in \overline{\Phi}\}\geq \textstyle\frac{1}{2(rM)^3}\cdot(\sum_{\gamma \geq 2}\sigma_\gamma(\theta^{(k)}))^2 = \alpha$, where we used that $\|A\|_* \leq \sqrt{r}\|A\|_{F}$ for a matrix $A$ with dimensions upper bounded by $r$, and that the probability of a single entry being observed when entries are observed uniformly is lower bounded by $1/r^2$.  Hence, these tensors satisfy \textbf{A5} by construction.

We lastly turn to the question of interpretation of the incoherence condition \textbf{A5}.  There is a large amount of incoherence in the above class of tensors because either two indices $j,q$ are decoupled because they lie in distinct facets or these indices jointly belong to the same facet that is decoupled from every other index.  Interpreted in this way, we can see why the example (\ref{eqn:tensor_ex}) displays pathological behavior: The value of an entry in the tensor $\psi$ is very sensitive to changes in $x_3$, and so the indices $1,2$ do not have sufficient incoherence from the index $3$ for our property \textbf{A5} to hold.

\subsection{Tensor Completion Algorithm for Low-Rank Ideal Partitions}
\label{section:tca}

As we have shown above, when \textbf{A1}--\textbf{A5} are satisfied, the risk gap $\mathcal{G}_{jq}$ is zero (non-zero) when the indices $j,q$ are decoupled (coupled) in the ideal partition $\Gamma^*$.  The idea of our algorithm is that we will use estimates of the risk gap $\hat{\mathcal{G}}_{jq}$ to construct an estimate of the ideal partition $\hat{\Gamma}$, and this will lead to a consistent estimation procedure because estimates of the risk gap converge significantly faster than estimates of the completed tensor.  Let $t_n$ be a threshold. The steps are:
\begin{enumerate}
\item Define the initial partition to be $\text{facets}(\hat{\Gamma}) = \{\{1\}\}$.  The remaining variables will be subsequently added to the partition.
\item For the variables indicated by $j = 2,\ldots,p$, do the following:
\begin{enumerate}
\item For the partitions represented by $k = 1,\ldots,\#\text{facets}(\hat{\Gamma})$
\begin{enumerate}
\item Let $q = (F_k)_1$ and compute the empirical version of the risk gap: $\hat{\mathcal{G}}_{jq} = \min\{\hat{\overline{R}}_{jq}(\overline{B})\ |\ \overline{B} \in \overline{\Phi}\} - \min\{\hat{R}_{jq}(B)\ |\ B \in \Phi\}$, where $\hat{R}_{jq}(B) = \frac{1}{n}\sum_{i=1}^n(-y\langle i\rangle\cdot\log (\beta_{x_j\langle i\rangle, x_q\langle i\rangle}) + \beta_{x_j\langle i\rangle, x_q\langle i\rangle})$ and also $\hat{\overline{R}}_{jq}(\overline{B}) = \frac{1}{n}\sum_{i=1}^n(-y\langle i\rangle\cdot\textstyle(\log \overline{\beta}_{x_j\langle i\rangle} + \log\overline{\beta}_{x_q\langle i\rangle}) + \overline{\beta}_{x_j\langle i\rangle}\overline{\beta}_{x_q\langle i\rangle})$.
\item If $\hat{\mathcal{G}}_{jq} > t_n$, then add $j$ to the $k$-th facet ($\hat{F}_k = \hat{F}_k \cup j$) and break this inner loop.
\end{enumerate}
\item If $j$ was not added to any facet, then add $j$ as its own facet ($\text{facets}(\hat{\Gamma}) = \text{facets}(\hat{\Gamma}) \cup \{j\}$).
\end{enumerate}
\item Compute $\hat{\Theta}$ by solving (\ref{eqn:convex}) with the partition $\hat{\Gamma}$ and then inverting the mapping $u_{\mathcal{X}_k} = \log \theta_{\mathcal{X}_k}$.
\end{enumerate}

Our first result on the consistency of this estimation procedure applies to cases in which we know the value of $\alpha$ and set the threshold to $t_n = \alpha/2$.  Technically, this result applies to any threshold $t_n = \alpha/\eta$ for any $\eta \in (0,1)$.

\begin{theorem}
\label{thm:const_est}
If $\mathbf{A1}$--$\mathbf{A5}$ are satisfied and $t_n = \alpha/2$, then with probability at least $1 - c_1n^{-1} - 2p^2\cdot\exp(-c_2n(\alpha/4-c_3r/\sqrt{n})^2)$ we have $0 \leq R(\hat{\Theta}) - R(\psi_x) \leq \sqrt{m \rho\log n/n}$, where constants $c_1,c_2,c_3 > 0$ depend on $\mu,M$.
\end{theorem}

\begin{proof}
Two types of mistakes can occur when estimating the ideal partition $\Gamma^*$ using the values $\hat{\mathcal{G}}_{jq}$: Either (i) $\mathcal{G}_{jq} = 0$ but $\hat{\mathcal{G}}_{jq} > \alpha/2$, or (ii) $\mathcal{G}_{jq} \geq \alpha$ but $\hat{\mathcal{G}}_{jq} \leq \alpha/2$.  Restated, a type (i) error does not occur if $|\hat{\mathcal{G}}_{jq} - \mathcal{G}_{jq}| < \alpha/2$, and a type (ii) error does not occur if $|\hat{\mathcal{G}}_{jq} - \mathcal{G}_{jq}| < \alpha/2$.  And because the estimation procedure is constructed such that the maximum number of $\hat{\mathcal{G}}_{jq}$ estimates that will be computed is $p(p-1)/2$, Proposition \ref{prop:concentration} implies $
\mathbb{P}(\max_{(j,q)\in\mathcal{J}}|\hat{\mathcal{G}}_{jq} - \mathcal{G}_{jq}| < \alpha/2) \geq \textstyle 1 - 2p^2\cdot\exp(-c_2n(\alpha/4-c_3r/\sqrt{n})^2)$, where $\mathcal{J}$ is the set of indices $(j,q)$ for which $\hat{\mathcal{G}}_{jq}$ is computed, and $c_2,c_3 > 0$ are constants that depend on $\mu,M$.  This expression is the probability that the estimated partition $\hat{\Gamma}$ is equal to the ideal partition $\Gamma^*$.

Let $\mathcal{A}$ be the event that $\hat{\Gamma} = \Gamma^*$, and let $\mathcal{B}$ be the event that $0 \leq R(\hat{\Theta}) - R(\Theta^*) \leq \sqrt{m \rho\log n/n}$.  Then $\mathbb{P}(\mathcal{B}) \geq \mathbb{P}[\mathcal{B}|\mathcal{A}]\cdot\mathbb{P}(\mathcal{A}) \geq (1 - 2p^2\cdot\exp(-c_2n(\alpha/4-c_3r/\sqrt{n})^2))\cdot(1-c_1n^{-1})$, which has the lower bound $1 - c_1n^{-1} - 2p^2\cdot\exp(-c_2n(\alpha/4-c_3r/\sqrt{n})^2)$.  The proof concludes by recalling that $R(\Theta^*) = R(\psi_x)$ by (\ref{eqn:bregman}).$\qquad$
\end{proof}

The value of $\alpha$ is not always known \emph{a priori}, and so we consider an alternative threshold that does not use the value of $\alpha$.  The downside of this alternative is that the results must necessarily be asymptotic because when $t_n > \alpha$, we cannot lower bound the probability of choosing the correct partition using the bounds from Proposition \ref{prop:concentration}, since these bounds only ensure that the estimation error lies within an interval.

\begin{theorem}
\label{thm:const_est2}
If $\mathbf{A1}$--$\mathbf{A5}$ are satisfied, $t_n = c_4/\sqrt{\log n}$ where $c_4 > 0$ is a constant, $r = O(1)$, and $\log p = o(n/\log n)$; then $R(\hat{\Theta}) - R(\psi_x) = O_p(\sqrt{m \rho\log n/n})$.
\end{theorem}

\begin{proof}
The proof is roughly the same as the proof of Theorem \ref{thm:const_est}, and so we highlight the main differences.  Since $t_n$ is strictly decreasing, there is some $N$ such that $t_n < \alpha/2$ for all $n \geq N$.  For the remaining arguments in the proof, we will assume $n \geq N$.  Next, note that the mistakes we can make are: Either (i) $\mathcal{G}_{jq} = 0$ but $\hat{\mathcal{G}}_{jq} > t_n$, or (ii) $\mathcal{G}_{jq} \geq \alpha$ but $\hat{\mathcal{G}}_{jq} \leq t_n$.  Restated, a type (i) error does not occur if $|\hat{\mathcal{G}}_{jq} - \mathcal{G}_{jq}| < t_n$, and a type (ii) error does not occur if $|\hat{\mathcal{G}}_{jq} - \mathcal{G}_{jq}| < \alpha/2$.   As a result, Proposition \ref{prop:concentration} implies $\mathbb{P}(\max_{(j,q)\in\mathcal{J}}|\hat{\mathcal{G}}_{jq} - \mathcal{G}_{jq}| < t_n) \geq \textstyle 1 - 2p^2\cdot\exp(-c_2n(\frac{1}{2}c_4/\sqrt{\log n}-c_3r/\sqrt{n})^2)$. And so $\mathbb{P}(\mathcal{B}) \geq \mathbb{P}[\mathcal{B}|\mathcal{A}]\cdot\mathbb{P}(\mathcal{A}) \geq \textstyle(1 - 2p^2\cdot\exp(-c_2n(\frac{1}{2}c_4/\sqrt{\log n}-c_3r/\sqrt{n})^2))(1-c_1n^{-1})$, which leads to the desired result.$\qquad$
\end{proof}

\subsection{Comparison of Statistical Convergence Rates}

These results imply we need $O((m\rho)^{1+\zeta})$, for any $\zeta > 0$, measurements to ensure $|R(\hat{\Theta}) - R(\psi_x)| = O_p(1)$.  Reshaping the tensor (as in \cite{mu2014}) may potentially improve this statistical rate, but this requires further study that is beyond the scope of the present paper.  Recall that existing convex optimization-based methods need $O(\pi^{\lfloor p/2\rfloor}r^{\lceil p/2\rceil})$ measurements whereas an NP-hard formulation needs $O(\pi^p + r\pi p)$ points \cite{mu2014}; note $\rho \geq \text{rank}_\oplus(\cdot) \geq \pi = \max_j \text{rank}_\boxplus(\cdot)$ \cite{desilva2008}, but these bounds are not tight \cite{landsberg2012}.  It is difficult to compare approaches because we use low effective dimension as a surrogate for tensor rank, while many existing methods use low multilinear rank \cite{tomioka2010,signoretto2010,gandy2011,liu2013,mu2014,zhang2014}.    

However, we can make a direct comparison in the special case of rank-1 tensors (where $\pi =  \text{rank}_\boxplus(\cdot) = 1$).  In this case, existing convex optimization-based approaches need $O(r^{\lceil p/2\rceil})$ measurements while our approach only needs $O((rp^2)^{1+\zeta})$, for any $\zeta > 0$, measurements (since $m \leq p$).  Restated, our approach requires a polynomial in $p$ number of measurements, while existing approaches need an exponential in $p$ number of measurements.  Also, our approach is essentially a quadratic factor away from the NP-hard formulation, which needs $O(rp)$ measurements in this case.  Methods based on the tensor nuclear norm need $O(r^{p/2})$ measurements \cite{yuan2015,yuan2016}, while methods based on the higher-order power method \cite{uschmajew2015} can achieve the information theoretic limit $O(rp)$ whenever the algorithm converges to a global optimum.

\subsection{Approximate Low-Rank Structure}
\label{section:aplrs}

So far we have assumed the ideal partition $\Gamma^*$ for $\psi$ has low effective dimension $\rho(\Gamma^*)$ and consequently describes a tensor with low rank; however, it is common to study estimation procedures for instances with approximate low-rank or sparsity structures (e.g., \cite{bickel2008,chatterjee2014}).  Unfortunately, it is unclear how to define approximate low-rank structure for the class of tensors we consider.  The difficulty is that our procedure works by exactly estimating the ideal partition $\Gamma^*$, but if a tensor approximately has low effective dimension then we would need to estimate an approximate partition $\overline{\Gamma}$.  However, partitions are discrete and so there is no clear notion of approximation.

Given these ambiguities with defining approximate low-rank structure, we consider a related notion: We will estimate tensors with low bias that are also low rank and have low effective dimension.  There is a tradeoff inherent in this between the amount of bias and the rank of the tensor.  Smaller bias will lead to higher rank tensors, while larger bias will lead to lower rank tensors.  It is difficult to analytically answer the question of how to control this tradeoff, and so instead we describe a cross-validation approach that can be used to control this.

The challenge with cross-validation is that we will need to control our effective dimension $\rho$; otherwise the cross-validation error will not be an accurate estimate of the actual loss.  We will create a finite sequence of nested partitions $\Gamma^1 \sqsubset \Gamma^2 \sqsubset \cdots \sqsubset \Gamma^q$, where $\Gamma^j \sqsubset \Gamma^{j+1}$ denotes that $F\in\text{faces}(\Gamma^{j+1})$ whenever $F \in\text{facets}(\Gamma^j)$.  The nested partitions will be constructed using a set of thresholds $T = \{t_1, t_2, \ldots, t_q\}$, and we will use cross-validation to pick the threshold.  Note that in general some subset of partitions may be equivalent (i.e., there may be $j$ such that $\Gamma^j = \Gamma^{j+1}$).

For simplicity, we will consider leave-$k$-out cross-validation with $k =  n/2$.  The corresponding tensor completion algorithm using cross-validation is:
\begin{enumerate}
\item For each $t_j \in T$, do the following:
\begin{enumerate}
\item Apply the algorithm from \S \ref{section:tca} to the full data set $(x\langle i\rangle, y\langle i\rangle)$ for $i = 1,\ldots,n$, to estimate the risk gaps $\hat{\mathcal{G}}_{jq}(t_j)$ and partitions $\Gamma^j = \hat{\Gamma}(t_j)$.
\item Use the data $(x\langle i\rangle, y\langle i\rangle)\text{ for } i = (\lfloor n/2\rfloor + 1),\ldots,n$ to compute estimates $\tilde{\Theta}(t)$ by solving (\ref{eqn:convex}) with the partition $\hat{\Gamma}(t)$ and then inverting the mapping $u_{\mathcal{X}_k} = \log \theta_{\mathcal{X}_k}$.
\item Compute the empirical cross-validation error $\hat{V}(t)$, which is defined as $\hat{V}(t) = \frac{1}{\lfloor n/2\rfloor}\sum_{i=1}^{\lfloor n/2\rfloor} (-y\langle i\rangle\cdot\textstyle\sum_{k = 1}^m\log \tilde{\theta}(t)_{\mathcal{X}_k\langle i\rangle} + \textstyle\prod_{k = 1}^m\tilde{\theta}(t)_{\mathcal{X}_k\langle i\rangle})$.
\end{enumerate}
\item Set $\hat{t} = \arg \min \{\hat{V}(t)\ |\ t \in T\}$ to be the threshold selected by cross-validation. 
\item Use the full data set $(x\langle i\rangle, y\langle i\rangle)$ for $i = 1,\ldots,n$, to compute the final estimate $\hat{\Theta}(\hat{t})$ by solving (\ref{eqn:convex}) with the partition $\hat{\Gamma}(\hat{t})$ and then inverting the mapping $u_{\mathcal{X}_k} = \log \theta_{\mathcal{X}_k}$.
\end{enumerate}

Suppose that $t^* = \arg \min \{R(\hat{\Theta}(t))\ |\ t \in T\}$ is the optimal threshold.  The following theorem shows that we can achieve an oracle inequality \cite{arlot2010} using leave-$k$-out cross-validation.  Note that we do not assume \textbf{A4},\textbf{A5} hold.

\begin{theorem}
\label{thm:cv}
If $\mathbf{A1}$--$\mathbf{A3}$ are satisfied, then with probability at least $1 - 10c_1n^{-1}\cdot(\#T)$ we have $R(\hat{\Theta}(\hat{t})) - R(\psi_x) \leq R(\hat{\Theta}(t^*)) - R(\psi_x) + (4\sqrt{2}+2)\sqrt{m_q\rho_q\log n/n}$, where $m_q = m(\Gamma^q)$, $\rho_q = \rho(\Gamma^q)$, and constant $c_1 > 0$ depends on $\mu,M$.
\end{theorem}

\begin{proof}
Observe that we must have
\begin{equation}
\label{eqn:ti0}
R(\hat{\Theta}(\hat{t})) - R(\psi_x) \leq R(\hat{\Theta}(t^*)) - R(\psi_x) + |R(\hat{\Theta}(\hat{t})) - R(\hat{\Theta}(t^*))|,
\end{equation}
and so applying the triangle inequality to the second term gives
\begin{equation}
\label{eqn:ti1}
|R(\hat{\Theta}(\hat{t})) - R(\hat{\Theta}(t^*))| \leq |R(\hat{\Theta}(\hat{t})) - \hat{V}(\hat{t})| + |\hat{V}(\hat{t}) - R(\hat{\Theta}(t^*))|.
\end{equation}
We will deal with the two terms on the right separately.  

Applying the triangle inequality to the first term of (\ref{eqn:ti1}) gives $|R(\hat{\Theta}(\hat{t})) - \hat{V}(\hat{t})| \leq |R(\hat{\Theta}(\hat{t})) - R(\Theta^*(\hat{t}))| + |R(\Theta^*(\hat{t})) - R(\tilde{\Theta}(\hat{t}))| + |R(\tilde{\Theta}(\hat{t})) - \hat{V}(\hat{t})|$.  The first two terms are bounded by Theorem \ref{thm:est_consistent}, and the third term is bounded by Proposition \ref{prop:concentration}.  So if we let $w_n = \sqrt{m_q\rho_q\log n/n}$, then using the union bound twice (once for having three terms and once for having multiple $t \in T$) gives 
\begin{equation}
\label{eqn:tin1}
|R(\hat{\Theta}(\hat{t})) - \hat{V}(\hat{t})| < (2\sqrt{2}+1)\cdot w_n,
\end{equation}
with probability at least $1 - 5c_1n^{-1}\cdot(\#T)$.

For the second term of (\ref{eqn:ti1}), the triangle inequality gives $|\hat{V}(\hat{t}) - R(\hat{\Theta}(t^*))| \leq |\hat{V}(\hat{t}) - R(\tilde{\Theta}(t^*))| + |R(\tilde{\Theta}(t^*)) - R(\Theta^*(t^*))| + |R(\Theta^*(t^*)) - R(\hat{\Theta}(t^*))|$.  The last two terms are bounded by Theorem \ref{thm:est_consistent}, and so we focus on the first term.  Because $\hat{t}$ minimizes $\hat{V}(t)$, we have
\begin{equation}
\label{eqn:ti4}
\hat{V}(\hat{t}) \leq \hat{V}(t^*) \Rightarrow \hat{V}(\hat{t}) - R(\tilde{\Theta}(t^*)) \leq \hat{V}(t^*) - R(\tilde{\Theta}(t^*)).
\end{equation}
Similarly, because $t^*$ is the minimizer of $R(\tilde{\Theta}(t))$, we have
\begin{equation}
\label{eqn:ti5}
R(\tilde{\Theta}(t^*)) \leq R(\tilde{\Theta}(\hat{t})) \Rightarrow \hat{V}(\hat{t}) - R(\tilde{\Theta}(\hat{t})) \leq \hat{V}(\hat{t}) - R(\tilde{\Theta}(t^*)).
\end{equation}
Combining (\ref{eqn:ti4}) and (\ref{eqn:ti5}) leads to $|\hat{V}(\hat{t}) - R(\tilde{\Theta}(\hat{t}))| \leq \max_{t \in T} |\hat{V}(t) - R(\tilde{\Theta}(t))|$, which can be bounded by Proposition \ref{prop:concentration}.  As a result, the union bound gives the following
\begin{equation}
\label{eqn:tin2}
|\hat{V}(\hat{t}) - R(\hat{\Theta}(t^*))| < (2\sqrt{2}+1)\cdot w_n,
\end{equation}
with probability at least $1 - 5c_1n^{-1}\cdot(\#T)$.  The result follows by using the union bound to combine (\ref{eqn:ti0}), (\ref{eqn:ti1}), (\ref{eqn:tin1}), and (\ref{eqn:tin2}).$\qquad$
\end{proof}

The lower bound on the success probability depends on the number of tuning parameters (via the cardinality of $T$), which is consistent with empirical results where using many tuning parameters leads to overfitting \cite{ng1997}.  Another note is we must control the decomposition complexity (by ensuring that $\rho_q$ is sufficiently small relative to $n$) to guarantee the above oracle inequality is achieved.  Lastly, this result implies that the cross-validation procedure is \emph{efficient} (in the sense of \cite{arlot2010}) when $\#T$ and $\rho_q$ grow sufficiently slowly in relation to $n$.

\section{Sparsity in Hierarchical Decompositions}

\label{section:shd}

Sparsity in the tensor $\psi_x$ can be used to improve the performance of our methods.  Here, sparsity means parameters $\theta_{\mathcal{X}_k}$ that are equal to 1, because this corresponds to a parameter not influencing the tensor value $\psi_x$.  In particular, we define a best sparse hierarchical decomposition as
\begin{equation}
\label{eqn:spnonconvex}
\hat{\Theta} = \arg\min \big\{\hat{R}(\Theta)\ \big|\ \Theta \in \Omega,\ \textstyle\sum_{k=1}^m\sum_{\mathcal{X}_k\in\mathcal{R}_k} |\log\theta_{\mathcal{X}_k}| \leq \lambda\big\}.
\end{equation}
The convex reparametrization is
\begin{equation}
\hat{U} = \arg\min \big\{\hat{R}(U)\ \big|\ U \in \Phi,\ \textstyle\sum_{k=1}^m\sum_{\mathcal{X}_k\in\mathcal{R}_k} |u_{\mathcal{X}_k}| \leq \lambda\big\}.
\end{equation}
Our convex reformulation matches the normal notion of coefficient sparsity because sparsity means the $u_{\mathcal{X}_k} = \log \theta_{\mathcal{X}_k}$ are equal to 0.  The $\textstyle\sum_{k=1}^m\sum_{\mathcal{X}_k\in\mathcal{R}_k} |u_{\mathcal{X}_k}| \leq \lambda$ constraint in the convex reformulation is just an $\ell_1$-norm inequality, and so it can be represented using a linear in $\rho$ number of linear inequalities using an LP lift \cite{yannakakis1991}.  Moreover, we can still solve this convex formulation in polynomial time; the proof is nearly identical to that of Proposition \ref{proposition:convex}, and so it is not included here.  

The key result regarding the above sparsity-exploiting formulations is an extension of Proposition \ref{prop:concentration}, from which we can then prove results analogous to those above for decomposition (Theorem \ref{thm:ranres}), approximation (Theorem \ref{thm:ranres2}), and completion (Theorems \ref{thm:const_est}, \ref{thm:const_est2}, and \ref{thm:cv}) of positive tensors.  We will not belabor this point by explicitly including these corresponding results or their proofs.  We instead prove only this key result on concentration of the empirical loss when $\Theta$ satisfies the following additional constraint: $\textstyle\sum_{k=1}^m\sum_{\mathcal{X}_k\in\mathcal{R}_k} |\log\theta_{\mathcal{X}_k}| \leq \lambda$.

\begin{proposition}
\label{prop:sparse}
Under $\mathbf{A1}$--$\mathbf{A3}$ and for any $\Gamma$, we have
\begin{equation}
\mathbb{P}\Big(\sup_{\Theta \in \Omega^\prime}\big|\hat{R}(\Theta) - R(\Theta)\big| < t\Big) \geq 1 - \textstyle \exp\Big(-C_3n\Big(t-C_4\lambda\sqrt{\frac{\log \rho}{n}}\Big)^2\Big),
\end{equation}
where $\Omega^\prime = \{\Theta \in \Omega : \|\log\Theta\|_1 \leq \lambda\}$, and constants $C_3,C_4 > 0$ depend on $\mu,M$.
\end{proposition}

\begin{proof}
The proof follows that of Proposition \ref{prop:concentration}, and so we only highlight the differences.  As before, we refer to the $\mathbbm{1}_{\mathcal{X}_k = \mathcal{X}_k\langle i\rangle}$ as pseudo-predictors, and the $u_{\mathcal{X}_k}$ are still the parameters.  If we define $\chi \in \{0,1\}^\rho$ to be the vector of pseudo-predictors, then $\|\chi\|_1 = m$, $\|\chi\|_2 = \sqrt{m}$, and $\|\chi\|_\infty = 1$.  The primary difference in this case is that the parameters belong to the modified set $\Phi^\prime = \{U \in \Phi : \|U\|_1 \leq \lambda\}$.  And so, results from \cite{kakade2009} immediately give that $\mathsf{R}(\mathsf{F}_\mathsf{W}) \leq \lambda\sqrt{2\log 2\rho/n}$. The result follows by combining this with (\ref{eqn:mcdiarmids}) and (\ref{eqn:radup}).$\qquad$
\end{proof}

This result shows that using soft-thresholding (via the $\|\log\Theta\|_1 \leq \lambda$ constraint) achieves performance that leverages the sparsity of the entries.  In particular, the statistical convergence rate (i.e., the upper bounds in Theorems \ref{thm:ranres}, \ref{thm:ranres2}, \ref{thm:const_est}, \ref{thm:const_est2}, and \ref{thm:cv}) implied by the proposition depends on effective dimension $\rho$ (rather than the total number of tensor entries $\prod r_i$) and on $\lambda$ (rather than the $\ell_1$-norm of parameters that are upper-bounded by $M$).  Expanding further, we would have the following convergence rates (i.e., the upper bounds in Theorems \ref{thm:ranres}, \ref{thm:ranres2}, \ref{thm:const_est}, \ref{thm:const_est2}, and \ref{thm:cv}) for positive tensor decomposition, approximation, and completion depending on the structure we leverage:\\

{\centering
\begin{tabular}{lc}
\hline
Structure & Convergence Rate \\
\hline
None & $O_p\Big(\hphantom{\lambda}\sqrt{\frac{r^p\log n}{n}}\quad\ \Big)$\\
Low Rank & $O_p\Big(\hphantom{\lambda}\sqrt{\frac{m\rho\log n}{n}}\ \ \ \Big)$\\
Sparse & $O_p\Big(\lambda\sqrt{\frac{\log r^p\log n}{n}}\Big)$\\
Sparse + Low Rank& $O_p\Big(\lambda\sqrt{\frac{\log \rho\log n}{n}}\ \Big)$\\
\hline
\end{tabular}\par}
\vspace{1\baselineskip}\vspace{-\parskip}
There is an additional statistical implication of sparsity in our framework.  Simultaneously regularizing for multiple sparse structures using convex approaches only regularizes with respect to the single most useful structure \cite{oymak2012,mu2014}.  In tensor completion, sparsity in the entries and the low-rank structure signify two distinct sparse structures.  Our framework overcomes the limitation of using only convex approaches to regularize with respect to these two sparse structures by combining hard- and soft-thresholding (similar to \cite{daspremont2007}).  In particular, we estimate a partition $\hat{\Gamma}$ using hard-thresholding applied to empirical risk gaps $\hat{\mathcal{G}}_{jq}$; and so we can use soft-thresholding to exploit sparsity in the coefficients of the partition decomposition. 

This combination of hard- and soft-thresholding allows our framework to handle other types of sparsity models.  One class of models \cite{goldfarb2014} is a low-rank tensor corrupted by a sparse additive perturbation.  Our approach can estimate a similar model: In particular, consider the model of a low-rank tensor corrupted by a sparse multiplicative perturbation.  Then, we can solve 
\begin{multline}
(\hat{\Theta},\hat{E}) = \arg\min \big\{\textstyle\frac{1}{n}\sum_{i=1}^n \big(-y\langle i\rangle\cdot\big(\log e_{\mathcal{X}\langle i\rangle} + \textstyle\sum_{k = 1}^m\log \theta_{\mathcal{X}_k\langle i\rangle}\big) + \\
\textstyle e_{\mathcal{X}\langle i\rangle}\cdot\prod_{k = 1}^m\theta_{\mathcal{X}_k\langle i\rangle}\big)\ \big|\ \Theta \in \Omega,\  \textstyle\sum_{\mathcal{X}\in\mathcal{R}} |\log e_{\mathcal{X}}| \leq \lambda\big\}.
\end{multline}
The convex reparametrization is
\begin{multline}
(\hat{U},\hat{W}) = \arg\min \big\{\textstyle\frac{1}{n}\sum_{i=1}^n \big(-y\langle i\rangle\cdot\textstyle\big(w_{\mathcal{X}\langle i\rangle} + \sum_{k = 1}^mu_{\mathcal{X}_k\langle i\rangle}\big) + \\
\textstyle\exp\big(w_{\mathcal{X}\langle i\rangle} + \sum_{k = 1}^mu_{\mathcal{X}_k\langle i\rangle}\big)\big|\ U \in \Phi,\  \textstyle\sum_{\mathcal{X}\in\mathcal{R}} |w_{\mathcal{X}}| \leq \lambda\big\},
\end{multline}
where we have the equivalence relation $w_{\mathcal{X}} = \log e_{\mathcal{X}}$.  We do not develop the theory for this (or other similar) models in this paper.

\section{Numerical Example}

\label{section:ne}

We compare our proposed estimators to three recent estimators for tensor completion.  More specifically, we compare five approaches:

\begin{remunerate}
\item The first estimator (\emph{Square Nuclear Norm} method) \cite{mu2014} is given by $\hat{\psi} = \arg\min_{\phi} \{\textstyle\frac{1}{n}\sum_{i=1}^n(y\langle i\rangle - {\phi}_{x\langle i\rangle})^2\ \ |\ \|\text{reshape}({\phi}_{(1)},\prod_{j=1}^{s}r_j, \prod_{j=s+1}^pr_j)\|_*\leq\lambda\}$, where ${\phi}_{(1)}$ is the unfolding of ${\phi}$ (into a matrix) along the first dimension \cite{desilva2008,mu2014}, the value $s$ minimizes $|\prod_{j=1}^{s}r_j - \prod_{j=s+1}^pr_j|$, $\text{reshape}(T,n_1,n_2)$ is a function that reshapes a matrix $T$ to have $n_1$ rows and $n_2$ columns, and $\lambda > 0$ is a constant.  

\item The second estimator (\emph{Maximum Nuclear Norm} method) \cite{zhang2014} is given by $\hat{\psi} = \min_{{\phi}} \{\textstyle\frac{1}{n}\sum_{i=1}^n(y\langle i\rangle - {\phi}_{x\langle i\rangle})^2\ |\ \max_j\{\|{\phi}_{(j)}\|_*\}\leq\lambda\}$, where ${\phi}_{(j)}$ is the unfolding ${\phi}$ along the $j$-th dimension, and $\lambda > 0$ is a constant.  

\item The third estimator (\emph{Alternating Least Squares} method) \cite{kolda2009} identifies a best \textsc{cp} decomposition \cite{kolda2009,hillar2013} by solving $\hat{\psi} = \min_{{\phi}} \{\textstyle\frac{1}{n}\sum_{i=1}^n(y\langle i\rangle - {\phi}_{x\langle i\rangle})^2\ |\ \phi = \sum_{j=1}^q v_1^j \otimes\cdots\otimes v_p^j\}$ using alternating least squares (ALS).  This is an ordinary least squares (OLS) problem in the variables $v_w^j$, for all $j\in[q]$, when the $v_k^j$, for all $k \in [p]\setminus w\wedge j\in[q]$, are fixed.  ALS minimizes this objective by iterating the index $w$ between $1,\ldots,p$ and solving the resulting OLS problems.  Though ALS is the most common approach for computing tensor decompositions \cite{kolda2009}, it typically converges to a local minimum \cite{kolda2009}.

\item The fourth estimator (\emph{Partition Log-Linear} method) is our tensor completion algorithm with cross-validation from \S \ref{section:aplrs} and with the estimator (\ref{eqn:nonconvex}).

\item The fifth estimator (\emph{Sparse Partition Log-Linear} method) is our tensor completion algorithm with cross-validation from \S \ref{section:aplrs} and with the estimator (\ref{eqn:spnonconvex}).

\end{remunerate}

The versions of the nuclear norm estimators we use are different from those in \cite{mu2014,zhang2014} because here we measure tensor entries with noise; the versions presented in \cite{mu2014,zhang2014} deal with the noiseless case.  Because we have noise, we instead minimize the deviation between measurements and estimates subject to a constraint that the nuclear norms are not large.  This is a common formulation for the noisy case of sparse estimation problems (see for instance \cite{osborne2000,greenshtein2004,candes2010}).  Also, we use a variant of the Maximum Nuclear Norm with simpler computation than \cite{zhang2014}, which converts the maximum into a smooth formulation that is amenable to specialized algorithm design.


Numerical implementations of the five estimators have been made available\footnote{\url{http://ieor.berkeley.edu/~aaswani/plrt/}}.  We implemented (i) our (Sparse) Partition Log-Linear method using the MATLAB toolbox for MOSEK \cite{mosek}, (ii) the Square and Maximum Nuclear Norm methods using the the CVX package \cite{grant08} for MATLAB, and (iii) the Alternating Least Squares method using MATLAB.  Our implementation code is not optimized for speed, and we have not studied the choice of algorithms for solving the convex reformulation of our estimators.  However, our estimators compute quickly because there are no constraints on matrix positive semidefiniteness (unlike the methods using nuclear norm).  We observed that our estimators computed faster than the nuclear norm estimators, but we do not include benchmarks because optimized code was not used.

The first numerical example consists of synthetic data generated from the tensor 
\begin{equation}
\psi = \begin{pmatrix}2 & 1 & 1\\1 & 2 & 1\\1 & 1 & 2\end{pmatrix}\otimes\begin{pmatrix}1\\2\\3\end{pmatrix}\otimes\begin{pmatrix}1\\1\\1\end{pmatrix}\otimes\begin{pmatrix}1\\1\\1\end{pmatrix},
\end{equation}
at two different noise levels, and we examine the estimation error as the amount of data increases for a fixed model.  The random variable $(1+z)$ has gamma distribution with shape $k > 0$ and scale $\theta > 0$.  A gamma distribution is used for the noise because it has support over $[0,\infty)$, and can be specified to ensure $\mathbb{E}(z) = 0$ as required by \textbf{A2}.  Though the unbounded support technically violates the assumption in \textbf{A2} on the boundedness of the noise, this boundedness is not a crucial assumption and can be relaxed (as we discussed earlier).  The numerical results support this conclusion.

Entries of the tensor are measured uniformly, and we used leave-$k$-out cross-validation with $k = n/2$ to select the tuning parameters of the different approaches.  Results for 100 repeated simulations are shown in Table \ref{table:synthetic}.  The table reports average prediction error under a square loss $\mathcal{E} = (\prod r_i)^{-1}\sum_{x \in \mathcal{R}} (\psi_x - \hat{\psi}_x)^2$.  The results indicate that our estimation procedure is competitive with existing approaches to tensor completion.  For each scenario, either the Partition Log-Linear or Sparse Partition Log-Linear approach has the lowest estimation error.

\begin{table}[t]
\footnotesize
\centering
\begin{tabular}{|lcccccc|}
\hline
\multicolumn{7}{|c|}{Gamma Distribution $k = 1$, $\theta=1$ (with Variance 1) \Large\phantom{0} }\\
&\multicolumn{6}{c|}{$n$}\\ 
& 10 &50 &100 &500 &1000 &5000\\
\hline
Square Nuclear Norm & \phantom{0}9.28 & \phantom{0}7.96 & 6.44 & 2.10 & 1.12 & 0.24\\
Maximum Nuclear Norm & \phantom{0}9.29 & \phantom{0}8.51 & 7.44 & 2.58 & 1.61 & 0.68\\
Alternating Least Squares & 54.10 & 10.43 & 4.00 & 1.39 & 1.16 & 0.11 \\
Partition Log-Linear & 16.46 & \phantom{0}4.17 & 2.43 & 0.35 & 0.16 & 0.03 \\
Sparse Partition Log-Linear & \phantom{0}4.67 & \phantom{0}3.13 & 2.60 & 0.31 & 0.16 & 0.03\\
\hline
\multicolumn{7}{c}{}\\[-1ex]
\hline
\multicolumn{7}{|c|}{Gamma Distribution $k = \frac{1}{5}$, $\theta=5$ (with Variance 5) \Large\phantom{0} }\\
&\multicolumn{6}{c|}{$n$}\\ 
& 10 &50 &100 &500 &1000 &5000\\
\hline
Square Nuclear Norm & \phantom{00}9.31 & \phantom{0}9.27 & \phantom{0}9.23 & 6.18 & 4.40 & 1.03\\
Maximum Nuclear Norm & \phantom{00}9.31 & \phantom{0}9.29 & \phantom{0}9.25 & 6.71 & 4.25 & 1.46\\
Alternating Least Squares & 278.78 & 78.10 & 49.47 & 3.65 & 1.92 & 1.12\\
Partition Log-Linear & \phantom{0}15.52 & \phantom{0}8.72 & \phantom{0}7.03 & 2.18 & 1.31 & 0.16\\
Sparse Partition Log-Linear & \phantom{00}5.10 & \phantom{0}4.08 & \phantom{0}3.90 & 2.67 & 1.27 & 0.16\\
\hline
\end{tabular}
\caption{Median of Average Estimation Error Over 100 Trials}
\label{table:synthetic}
\end{table}

\section{Regression with Categorical Variables}

\label{section:rcv}

We refer to a model with purely categorical predictors and a numeric response as a \emph{combinatorial regression model}.  In particular, suppose there are $p$ categorical predictors.  For the $j$-th predictor with $r_j$ different categories, we can assign each category to a unique integer in $[r_j]$.  With this notation, a combinatorial regression model can be written as $\psi_x := \mathbb{E}[y|x] = \theta_{x_1,x_2,\ldots,x_p}$, where $x_j \in [r_j]$ is the $j$-th predictor, and $\theta : [r_1] \times \ldots \times [r_p] \rightarrow \mathbb{R}$ is a function from a discrete to continuous space.

Even without knowledge about $\theta$, we can still identify the model because of the finiteness of the predictor space.  Specifically, if we have measurements $(x\langle i\rangle,y\langle i\rangle)$ for $i=1,\ldots,n$, then we can identify the model using $\hat{\psi}_x = \frac{1}{|\mathcal{I}(x)|}\sum_{i \in \mathcal{I}(x)} y\langle i\rangle$, where $\mathcal{I}(x) = \{i : x\langle i\rangle = x\}$.  Under typical assumptions on noise, this is a consistent estimator.  Unfortunately, its convergence rate $O_p(r^p/n)$, where $r = \max r_j$, is exponentially slow in $p$; this is not surprising because there is a combinatorial explosion that leads to a curse of dimensionality if we try to estimate each value of $\psi_x$ separately.
 
The standard approach to reducing dimensionality is to (i) define \emph{coding variables} (e.g., dummy predictors) to convert categorical variables into numerical values, and (ii) perform regression using the coding variables \cite{cohen2013}.  Though this converges at $O_p(rp/n)$, this can be restrictive because the impact of different predictors $x_j,x_k$ for $j\neq k$ is completely decoupled, which is not reflective of a combinatorial model.  (In principal, variables can be coupled by defining pairwise (or higher) coding variables, but this is typically done using domain knowledge.)  


Our notation for a combinatorial regression model is suggestive of another interpretation of low-rank structure: We propose the novel interpretation that a combinatorial regression model can be represented by a low-rank tensor.  The tensor is indexed by the $x_j$, which are integers.  Thus, the problem of estimating a combinatorial regression problem is equivalent to a noisy low-rank tensor completion problem.  Given the discussion of tensor completion in this paper, we can consider an example of identifying a combinatorial regression model with real data using tensor completion.

\subsection{Violacein Pathway}

\begin{figure}
\centering
 \subfloat[Square Nuclear Norm]{\includegraphics[trim=0.20in 0 0.23in 0, clip]{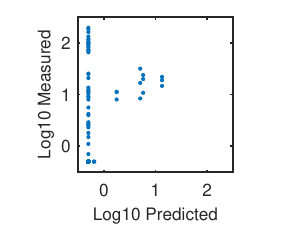}}
 \subfloat[Maximum Nuclear Norm]{\includegraphics[trim=0.20in 0 0.23in 0, clip]{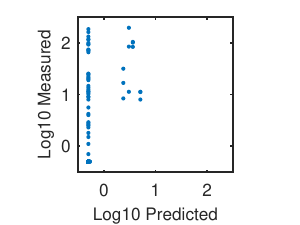}}
 \subfloat[Dummy Coding Linear]{\includegraphics[trim=0.20in 0 0.23in 0, clip]{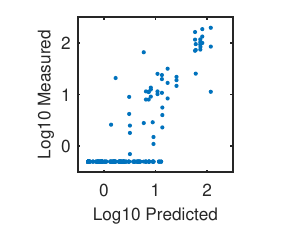}}\\
 \subfloat[Alternating Least Squares]{\includegraphics[trim=0.2in 0 0.23in 0, clip]{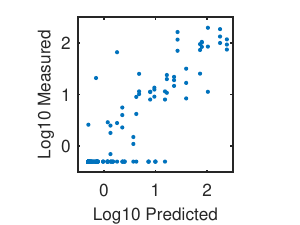}}\quad
 \subfloat[(Sparse) Partition Log-Linear]{\includegraphics[trim=0.2in 0 0.23in 0, clip]{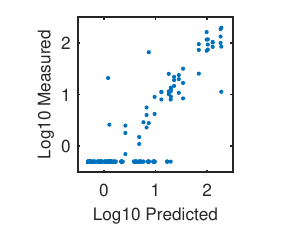}}
\caption{Comparison between predicted and measured violacein production levels.}
\label{fig:vio}
\end{figure}

Bioengineered metabolic pathways hold promise for the production of pharmaceuticals and transportation fuels, and they are constructed in a combinatorial fashion by varying different discrete design elements.  This combinatorial nature makes it challenging to engineer the pathway to maximize production of the bioproduct, and so one proposed idea is to (i) construct a model relating design parameters to the amount of bioproduct produced, and then (ii) use this model to determine which combination of design elements maximizes the bioproduct \cite{lee2013}.

Bioengineered pathways can be represented by a combinatorial regression model, and so it is instructive to apply tensor completion methods.  In the pathway studied in \cite{lee2013}, there are five predictors $p = 5$, and each predictor has five levels $r_j = 5$ for $j = 1,\ldots,5$.  The data is categorized into either a training data set or a validation data set, and each respective data set consists of different experiments with explicitly different predictor values (i.e., design elements) used for each; the validation data set was constructed to be a true validation data set for the original model in \cite{lee2013}.  

A comparison of predicted and measured values for models computed using different approaches is shown in Figure \ref{fig:vio}.   Sparse Partition Log-Linear is not separately shown because cross-validation chose $\lambda$ to make the model identical to Partition Log-Linear.  All models were constructed using data designated as the training set in \cite{lee2013}, and the predictions and measured values in Figure \ref{fig:vio} correspond to data designated as the validation set in \cite{lee2013}.  The equipment could not measure values smaller than 0.5, and so measured values and model predictions smaller than this were set to 0.5.  

The predictions of (Sparse) Partition Log-Linear most closely match the measured values.  Dummy Coding Linear (i.e., model in \cite{lee2013} using \cite{aswani2011}) and Alternating Least Squares perform less well, and Square Nuclear Norm and Maximum Nuclear Norm do not work well for this data.  Quantitatively, Spearman's rank correlation coefficient is interesting because we are interested in models that can predict the relative amount of bioproduct for a particular combinatorial design.  The Spearman correlation coefficient (for measurements above the minimum detectable level of 0.5) is 0.84 for (Sparse) Partition Log-Linear, 0.80 for Alternating Least Squares, 0.75 for Dummy Coding Linear, 0.17 for Maximum Nuclear Norm, and -0.26 for Square Nuclear Norm.  We also conducted a bias-corrected bootstrap hypothesis test \cite{efron1987} to determine whether the model fit improvement of (Sparse) Partition Log-Linear (as compared to a model computed by the other methods) was statistically significant: This hypothesis test returned (p=0.009) for Alternating Least Squares, (p\textless0.001) for Dummy Coding Linear, (p\textless0.001) for Maximum Nuclear Norm, and (p\textless0.001) for Square Nuclear Norm.  These results indicate the model fit improvement of (Sparse) Partition Log-Linear is statistically significant.

\section{Conclusion}

We defined a new decomposition for positive tensors, showed it can be computed in polynomial time using a randomized algorithm, and justified the design of this decomposition and its loss function by identifying important cases where they coincide with the usual tensor decomposition and squared loss function.  We extended this framework to tensor completion, and showed our approach has improved statistical performance in comparison to existing approaches.  We provided a novel interpretation of regression problems with categorical variables as tensor completion problems, and numerical examples with synthetic data and data from a bioengineered metabolic network displayed the improved performance of our approach.  Our current work includes application of our approach to regression with categorical variables on larger data sets, and we have found our approach computationally scales well in terms of the number of data points $n$, tensor order $p$, and dimension $r$.  However, scalability issues arise as the effective dimension $\rho(\Gamma) = \sum_{k=1}^m \prod_{j \in F_k} r_j$ (which has exponential size in terms of the cardinality of $F_k$) grows.  It may be interesting to develop specialized optimization algorithms to solve our formulation in such settings.

\section*{Acknowledgements}
The author thanks John E. Dueber and Michael E. Lee for providing the violacein data set \cite{lee2013}.

\bibliographystyle{siam}
\bibliography{pos_lowten}

\end{document}